%% file: main.tex
\documentclass[reqno]{amsart}

\usepackage{amssymb}
\usepackage{amsthm}
\usepackage{amsmath}
\usepackage{enumitem}
\usepackage{xpatch}
\usepackage{hyperref}

\usepackage{tikz-cd}
\usepackage{leftindex}
\usepackage[capitalise]{cleveref}
\usepackage[all,2cell]{xy}
\UseAllTwocells
\usepackage{bm} 

\usepackage{marginnote}

\title{Groupoidal polygraphic homology} 

\author{Léonard Guetta}
\address{L\'eonard Guetta, Universiteit Utrecht, Utrecht Science Park,
Hans Freudenthalgebouw, Budapeestlaan 6, Utrecht}
\email{l.s.guetta@uu.nl}

\author{Fran\c{c}ois M\'etayer}
\address{Fran\c{c}ois M\'etayer, Universit\'e Paris Cit\'e, CNRS, IRIF, F-75013, Paris, France \&
Universit\'e Paris Nanterre}
\email{metayer@irif.fr}

\input{macro}

\begin{document}

\begin{abstract} 
We show that for a $1$-category $C$, the \ook k-polygraphic homology
of $C$ for any $k\geq 1$, that is taken with cofibrant resolutions in strict \ook
k-categories, does not depend on $k$ and is canonically isomorphic to
the homology of the classifying space of $C$. When $C$ is a groupoid,
we also show this for $k=0$. In particular, this means that the
classical homology of groups can be obtained by taking cofibrant
resolutions in strict \oo-groupoids.
In order to show these results, we develop the theory of discrete
Conduché fibrations in the category of strict \ook k-categories,
building on previous work by the first-named author.
\end{abstract}

\maketitle

\input{intro}
\input{basics} 
\input{syntax} 
\input{conduche} 
\input{resolproj} 
\input{polhom} 
\input{htpykan} 
\input{comparison} 
\bibliographystyle{alpha}
\bibliography{biblio}

\end{document}

%% file: macro.tex
\swapnumbers
\theoremstyle{plain}
\newtheorem{theorem}[subsubsection]{Theorem}
\newtheorem{corollary}[subsubsection]{Corollary}
\newtheorem{proposition}[subsubsection]{Proposition}
\newtheorem{lemma}[subsubsection]{Lemma}

\theoremstyle{definition}
\newtheorem{definition}[subsubsection]{Definition}

\theoremstyle{remark}
\newtheorem{remark}[subsubsection]{Remark}
\newtheorem{paragr}[subsubsection]{} 

\crefname{theorem}{Theorem}{Theorems}
\crefname{corollary}{Corollary}{Corollaries}
\crefname{proposition}{Proposition}{Propositions}
\crefname{lemma}{Lemma}{Lemmas}

\crefname{definition}{Definition}{Definitions}
\crefname{example}{Example}{Examples}

\crefname{rem}{Remark}{Remarks}
\crefname{paragr}{}{}

\newcommand{\nmbr}[1]{\mathbb #1}
\newcommand{\N}{\nmbr{N}}
\newcommand{\Z}{\nmbr{Z}}

\newcommand{\M}{\mathcal{M}}

\newcommand{\set}[1]{\left\{#1\right\}} 
\newcommand{\setof}[2]{\left\{ #1\; | \; #2\right\}}
\newcommand{\pair}[2]{(#1,#2)}       

\newcommand{\ilim}{\varinjlim} 
\newcommand{\source}{\mathsf{s}} 
\newcommand{\target}{\mathsf{t}} 
\newcommand{\cosource}{\sigma} 
\newcommand{\cotarget}{\tau} 
\newcommand{\geninc}{\mathsf{i}} 
\newcommand{\unit}[1]{1^{#1}} 
\newcommand{\Unit}[1]{\mathsf{1}_{#1}}

\newcommand{\homset}[3]{#1(#2,#3)} 
\newcommand{\fun}[2]{{#2}^{#1}} 
\newcommand{\Ob}{\mathrm{Ob}} 

\newcommand{\op}{{\mathrm{op}}} 
\newcommand{\tr}[2]{#1/#2} 
\newcommand{\wsk}{\ast} 

\newcommand{\ctg}[1]{\mathbf{#1}}  
\newcommand{\Set}{\ctg{Set}}    
\newcommand{\Cat}{\ctg{Cat}}   
\newcommand{\nCat}[1]{\Cat_{#1}}  
\newcommand{\npCat}[2]{\Cat_{#1,#2}}  
\newcommand{\opCat}[1]{\Cat_{\omega,#1}} 
\newcommand{\funopCat}[2]{\Cat_{\omega,#1}^{#2}} 
\newcommand{\nCatp}[1]{\Cat_{#1}^{+}} 
\newcommand{\npCatp}[2]{\Cat_{#1,#2}^{+}} 
\newcommand{\oCat}{\nCat{\omega}} 
\newcommand{\glob}{\ctg{G}}
\newcommand{\Glob}{\ctg{Glob}} 
\newcommand{\nGlob}[1]{\Glob_{#1}} 
\newcommand{\oGlob}{\nGlob{\omega}} 
\newcommand{\Pol}{\ctg{Pol}} 
\newcommand{\nPol}[1]{\Pol_{#1}} 
\newcommand{\npPol}[2]{\Pol_{#1,#2}} 
\newcommand{\opPol}[1]{\npPol{\omega}{#1}} 
\newcommand{\Grpd}{\ctg{Grpd}}
\newcommand{\nGrpd}[1]{\Grpd_{#1}} 
\newcommand{\oGrpd}{\nGrpd{\omega}} 
\newcommand{\Ch}[1]{\ctg{Ch}_{#1}} 
\newcommand{\Chp}{\Ch{\geq 0}} 
\newcommand{\Ab}{\ctg{Ab}} 
\newcommand{\Term}{\ctg{1}} 

\newcommand{\RMod}[1]{\mathrm{Mod}_{#1}} 
\newcommand{\LMod}[1]{{}_{#1}\mathrm{Mod}} 

\newcommand{\functor}[1]{\mathsf{#1}} 
\newcommand{\trk}[1]{\functor{tr}_{#1}}
\newcommand{\inc}[2]{\functor{inc}_{#1}^{#2}}
\newcommand{\oinc}[1]{\inc{#1}{}}
\newcommand{\emb}[2]{\functor{j}_{#2}^{#1}}
\newcommand{\pifun}[2]{\functor{\pi}_{#2}^{#1}}
\newcommand{\maxgr}[2]{\functor{\iota}_{#2}^{#1}}
\newcommand{\skl}[1]{\functor{sk}_{#1}}
\newcommand{\oemb}[1]{\emb{#1}{}}
\newcommand{\opifun}[1]{\pifun{#1}{}}
\newcommand{\omaxgr}[1]{\maxgr{#1}{}}
\newcommand{\fgf}[1]{\functor{U}_{#1}} 
\newcommand{\frf}[1]{\functor{F}_{#1}} 

\newcommand{\abel}{\functor{\lambda}} 
\newcommand{\abelk}[1]{\functor{\abel}^{#1}}

\newcommand{\B}{\mathbb{B}} 

\newcommand{\homnat}{\theta}

\newcommand{\doubl}[2]{\ar@<2pt>[l]^{#2}\ar@<-2pt>[l]_{#1}}

\newcommand{\doubr}[2]{\ar@<2pt>[r]^{#2}\ar@<-2pt>[r]_{#1}}

\newcommand{\doubld}[2]{\ar@<2pt>[ld]^{#2}\ar@<-2pt>[ld]_{#1}}

\newcommand{\tdouble}[3]{
  \arrow[#1,shift left=0.4ex,#3]
  \arrow[#1,shift right=0.4ex,#2]
}

\newcommand{\comp}[1]{\ast_{#1}} 
\newcommand{\sce}[1]{\source_{#1}} 
\newcommand{\tge}[1]{\target_{#1}} 

\newcommand{\inv}[1]{{#1}^{-1}} 

\newcommand{\cosce}[1]{\cosource_{#1}} 
\newcommand{\cotge}[1]{\cotarget_{#1}} 
\newcommand{\frcat}[1]{#1^*} 
\newcommand{\frgp}[1]{#1^\top} 

 \newcommand{\oo}{\ensuremath{\omega}}
 \newcommand{\ook}[1]{\ensuremath{\pair{\omega}{#1}}}
 \newcommand{\pk}[2]{\ensuremath{\pair{#1}{#2}}} 
\newcommand{\clx}[2]{\pair{#1}{#2}}
\newcommand{\gni}[1]{\geninc_{#1}}

\newcommand{\Sph}[1]{\mathbb{S}_{#1}} 
\newcommand{\Dsk}[1]{\mathbb{D}_{#1}} 
\newcommand{\gcof}[1]{\functor{i}_{#1}} 
\newcommand{\gray}{\otimes}
\newcommand{\grayk}[1]{\underset{#1}{\gray}} 

\newcommand{\expr}[1]{\mathcal{E}[#1]} 
\newcommand{\WW}{\mathcal{W}}
\newcommand{\clWW}{\fcls{\WW}}
\newcommand{\wtx}[1]{\WW[#1]} 
\newcommand{\wtxr}[2]{\WW_{#1}[#2]} 
\newcommand{\wtxeq}[1]{\clWW[#1]} 
\newcommand{\formal}[1]{\mathsf{#1}} 
\newcommand{\fcst}[1]{\formal{c}_{#1}} 
\newcommand{\finv}[1]{\formal{c}^{\boldsymbol{-}}_{#1}}
\newcommand{\fid}[1]{\formal{i}_{#1}} 
\newcommand{\fcomp}[1]{\bm{\star}_{#1}} 
\newcommand{\type}[3]{#1:#2\to #3} 
\newcommand{\frel}{\leadsto} 
\newcommand{\cfrel}{\leadsto_1} 
\newcommand{\freq}{\simeq} 
\newcommand{\cans}[1]{\rho^{#1}} 
\newcommand{\fcls}[1]{\widetilde{#1}}
\newcommand{\finc}[1]{\formal{c}^{#1}}

\newcommand{\scex}[1]{\extend{\source}_{#1}}
\newcommand{\tgex}[1]{\extend{\target}_{#1}}

\newcommand{\scecl}[1]{\fcls{\source}_{#1}}
\newcommand{\tgecl}[1]{\fcls{\target}_{#1}}

\newcommand{\lgh}{\ell}

\newcommand{\WWo}{\mathcal{W}_*}
\newcommand{\clWWo}{\fcls{\WWo}}
\newcommand{\wtxo}[1]{\WWo[#1]} 
\newcommand{\wtxeqo}[1]{\clWWo[#1]} 
                              
\newcommand{\tens}[1]{\underset{#1}{\otimes}} 
\newcommand{\Ltens}[1]{\overset{\LL}{\tens{#1}}} 
\newcommand{\Tor}[2]{\functor{Tor}^{#1}_{#2}} 
\newcommand{\cst}[1]{\underline{#1}} 
\newcommand{\Ho}{\functor{H}} 
\newcommand{\LL}{\mathbb{L}} 
\newcommand{\RR}{\mathbb{R}} 
\newcommand{\Hop}[1]{\Ho^{#1,\mathrm{pol}}} 
\newcommand{\Hopol}{\Ho^{\mathrm{pol}}}

\newcommand{\Derp}[1]{\functor{D}_{\geq 0}(#1)}
\newcommand{\Loc}[1]{\mathrm{Ho}(#1)} 

\newcommand{\GZ}[1]{\overline{#1}} 

\newlength{\mylen}
  \NewCommandCopy{\hocolim}{\varinjlim}
  \makeatletter
  \NewCommandCopy{\holim@}{\varlim@}
  \xpatchcmd{\hocolim}{\varlim@}{\holim@}{}{}
  \xpatchcmd{\holim@}{lim}{holim}{}{}
  \makeatother

\newcommand{\nbd}{\nobreakdash} 
\newcommand{\extend}[1]{\overline{#1}} 
\newcommand{\cndi}{\textbf{(c1)}} 
\newcommand{\cndii}{\textbf{(c2)}} 


%% file: intro.tex
\section{Introduction}\label{sec:intro}

The category $\oCat$ of strict, globular $\oo$-categories bears a
``canonical'' model structure introduced
in~\cite{lafontetal:folkms}. In this structure, the cofibrant objects
are of the form $\frcat P$ where $P$ is a polygraph
(see~\cite{abgmmm:polybk}) and $\frcat P$ denotes the \oo-category
freely generated by $P$. Let $\Chp$ be the category of chain complexes
of abelian groups in non-negative degree equipped with its usual
projective model structure. There is a left Quillen abelianization
functor $\abel:\oCat\to \Chp$ (see~\ref{subsec:abel} below). Let $\LL\abel:\Loc{\oCat}\to\Loc{\Chp}$ denote the
left derived functor of $\abel$ between the corresponding homotopy
categories. The {\em polygraphic homology} of an \oo-category $C$
is by definition
\[\Hopol(C)=\LL\abel(C).\]
To compute this homology, we first produce
a cofibrant replacement of $C$, namely a polygraph $P$ together with 
a trivial fibration $p:\frcat P\to C$, then apply the functor $\abel$
directly to $\frcat P$, whence the name.

Now for each integer $k\geq 0$, an \ook k-category is an \oo-category
in which all $i$-cells are invertible if  $i>k$. The full
subcategory of $\oCat$ whose objects are \ook k-categories is denoted
by $\opCat k$. For instance $\opCat 0$ is just the category of strict
\oo-groupoids.
The canonical inclusion functor $\oemb k:\opCat k\to\oCat$ admits a left
adjoint $\opifun k:\oCat\to\opCat k$. It turns out that the canonical model 
structure on $\oCat$ transfers to $\opCat k$ via $\oemb k$ and that
$\opifun k$ is again left Quillen. On the other hand, the functor
$\abel$ easily factors through $\opifun k$. Thus we get an
abelianization functor $\abelk k:\opCat k\to\Chp$ such that
$\abel=\abelk k\circ\opifun k$, and $\LL\abel$ is naturally isomorphic
to $\LL\abelk k\circ \LL\opifun k$. This leads to defining the
polygraphic homology of an \ook k-category $C$ by
\[\Hop{k}(C)=\LL\abelk k(C).\]
On the other hand, each \ook k-category $C$ can be seen as an \ook
{k'}-category for any $k'$ such that $k\leq k'\leq \oo$ and therefore also has an \ook{k'}-homology $\Hop{k'\!}(C)$. In particular, for
$k'=\oo$, we recover its polygraphic homology $\Hopol(C)$, as shown in 
the following diagram:
\[
  \begin{tikzcd}
    \opCat k \ar[r,"\oemb k"]\ar[d]& \oCat\ar[r,"\opifun k"]\ar[d] & \opCat k\ar[r,"\abelk k"]\ar[d] & \Chp\ar[d]\\
  \Loc{\opCat k}\ar[r,"\RR\oemb k"']& \Loc{\oCat}\ar[r,"\LL\opifun k"']\ar[ru,Rightarrow] & \Loc{\opCat
  k}\ar[r,"\LL\abelk k"']\ar[ru,Rightarrow]& \Loc{\Chp}.
  \end{tikzcd}
  \]
As the model structure on $\opCat k$ is right-induced by $\oemb k$,
the left-hand square commutes on the nose. However $\LL\opifun k\circ
\RR\oemb k$ is {\em not} the identity on $\Loc{\opCat k}$, but has
only a natural transformation $\epsilon:\LL\opifun k\circ
\RR\oemb k\to \Unit{\Loc{\opCat k}}$. This yields a natural
transformation from $\LL\abel\circ\RR\oemb k$ to $\LL\abelk k$. In
other words, for each \ook k-category $C$, there is a canonical morphism
\begin{equation}
  \label{eq:natural}
  \homnat_{C}\colon\Hopol(C)\to \Hop k (C),
\end{equation}
natural in $C$.
The present work addresses the following question:
\begin{quote}
 {\em In which cases is the morphism~(\ref{eq:natural}) an isomorphism?}
\end{quote}
 A very simple example
already shows that the answer is by no means trivial: consider $\Z$ as an
\ook 0-category $C$ where $C_0$ is a single point, $C_1=\Z$ and all
$i$-cells are identities for $i>1$. As an object of $\opCat 0$, $C$ is
cofibrant and therefore $\Hop 0(C)$ is just $\abelk 0(C)$, that is
$\Z$ concentrated in degrees $0$ and $1$. However,
$C$ is {\em not} a cofibrant object in $\oCat$, and to compute
$\Hopol(C)$ we first need a polygraphic resolution $\frcat P\to C$. Now, for any such $P$ and each $k>1$, the set of $k$-dimensional
generators must be non-empty and the direct computation of the homology of
$\abel\frcat P$ is far from obvious.

The main result of this paper is that~(\ref{eq:natural}) is an
isomorphism in the following two cases:
\begin{itemize}
\item $C$ is a $1$-category, seen as an \ook k-category for $k\geq 1$;
  \item $C$ is a groupoid, seen as an \ook k-category for $k\geq 0$.
  \end{itemize}
In fact, if $C$ is a $1$-category or a groupoid,  we take advantage of the existence
of the usual homology $\Ho(C)$ obtained by deriving the tensor product
of $C$-modules. Precisely, the
category $\RMod{C}$ of right $C$\nbd-modules is the category of
functors \[M \colon C^{\op} \to \Ab\] from $C^{\op}$ to the
  category of abelian groups and natural transformations between them,
  and the category $\LMod{C}$ of left $C$\nbd-modules is the category
  of functors \[N \colon C \to \Ab.\] Given a right $C$\nbd-module $M$ and a
  left $C$\nbd-module $N$, we can then define their total tensor
  product as the following coend
  \[
    M \tens{C}N := \int^{c \in C}M(c)\tens{\Z} N(c).
  \]
  This defines a functor $\RMod{C}\times \LMod{C} \to \Ab$,
   right exact in each variable and thus left derivable since the
  categories of left and right $C$\nbd-modules have enough
  projectives. This allows us to define the (classical) homology of $C$
  as
  \[
    \Ho(C):=\cst{\Z}\Ltens{C}\cst{\Z},
  \]
  where $\Ltens{C}$ is the total left derived functor of $\tens{C}$,
  and $\cst{\Z}$ denotes both the constant left $C$\nbd-module and the
  constant right
  $C$\nbd-module with value $\Z$. Note that if $C$ is a
  group, this definition is the usual definition of the homology
  of $C$ with coefficients in $\Z$.

  Now, if $1\leq k\leq \oo$ and $C$ is a $1$-category seen as an  \ook
  k-category, we construct a canonical morphism
  \begin{equation}
    \label{eq:poltoclassical}
     \Hop k (C) \to \Ho(C)
   \end{equation}
   natural in $C$ and show that it is an
   isomorphism. Likewise, if $C$ is a groupoid, the same holds
   for $0\leq k\leq \oo$.

 This is a significant result on its own: it says that the
  standard homology of $C$, that is $\Ho(C)$, is canonically
  isomorphic to its polygraphic homology computed using polygraphic
  resolutions in $\opCat{k}$ for any chosen $k\geq 1$ or even $k=0$ if $C$ is
  a groupoid. For example, for a group $G$, this means that we can
  compute the singular homology of the Eilenberg--MacLane space
  $K(G,1)$, using \emph{strict} $\oo$\nbd-groupoids.

  We then complete the argument by showing that, in the case of
  $1$-categories, the map~(\ref{eq:natural}) fits into a commutative triangle
 \[
    \begin{tikzcd}[column sep=small]
      \Hopol(C) \ar[dr]\ar[rr,"(1)"]&&\Hop k (C)\ar[dl]\\
      &\Ho(C)&
    \end{tikzcd}
  \]
  where both downward arrows are instances of the
  isomorphism~(\ref{eq:poltoclassical}). Therefore the horizontal
  arrow is also an isomorphism. When $C$ is a groupoid, the same
  result extends to the value $k=0$.

  Partial results about the isomorphism~\eqref{eq:poltoclassical} have
already been established in earlier works. The case $k=\oo$ is the main result of
\cite{guetta:homcat}. For $k=1$, a proof can be
extracted from Guiraud and Malbos' work \cite[Theorem
5.4.3]{guiraudmalbos:higdns}, up to naturality issues not addressed there.

The present article follows the same strategy as~\cite{guetta:homcat}
in proving that~\eqref{eq:poltoclassical} is an isomorphism and relies on two
 auxiliary results, which are interesting in their own right.

 First, we prove that any $1$-category $C$ with a terminal object has a
 trivial polygraphic homology as an \ook k-category for any $1 \leq k
 \leq \oo$. This means that for any such $k$, we have a natural isomorphism
   \begin{equation}\label{eq:trivhlgy}
      \Hop k (C) \cong \Z,
    \end{equation}
    where $\Z$ is considered as a chain complex concentrated in degree
    $0$. If $C$ is a groupoid, then this also works for $k=0$.

    Second, we prove that for any $1$-category $C$, the colimit
    \begin{equation}\label{eq:htpycolimintro}
      C \cong \ilim_{c \in C} \tr{C}{c}
    \end{equation}
    is in fact a \emph{homotopy colimit} in $\opCat{k}$ with respect to the
    canonical model structure for any $1 \leq k \leq \oo$, and if $C$ is a
    groupoid this also works for $k=0$. 

The existence of the natural
isomorphism~\eqref{eq:trivhlgy} ultimately follows from
the monoidal properties of the canonical model structure on $\opCat{k}$
with respect to the Gray tensor product. On the other hand, the proof
that~\eqref{eq:htpycolimintro} is a homotopy colimit is more elaborate and
relies on the theory of discrete Conduché fibrations for strict \ook k-categories, extending the case $k=\oo$ treated in
\cite{guetta:poldcf}. The key result is that if \[f \colon D
\to C\] is a discrete Conduché fibration between \ook k-categories and
if $C$ is free on an \ook k-polygraph $P$, then $D$ is also free on an
\ook k-polygraph $Q$ constructed functorially out of $P$. 

\subsection*{Organization of the paper}
After some preliminaries on strict \ook k-categories in
Section \ref{sec:basics}, the paper is divided into two main parts which are mostly
independent: Sections \ref{sec:syntax} and \ref{sec:basislift} develop the theory of discrete Conduché fibrations and Sections
\ref{sec:projresol}, \ref{sec:polhom}, \ref{sec:htpykan} and \ref{sec:comparison} are dedicated to the homotopical and homological
results. Readers willing to accept Theorem \ref{thm:main} can safely
skip the first part.
\subsection*{Acknowledgments}
The authors would like to thank Philippe Malbos for useful conversations about polygraphic resolutions for $\ook 1$\nbd-categories.
\subsection*{AI disclosure} Anthropic's Claude Sonnet 5 and OpenAI's
GPT-5.6 Luna were used for proofreading at the final stage of writing this paper, specifically to catch typos and English grammar
mistakes. No mathematical content was produced by AI.


%% file: basics.tex
\section{Basic notions on\texorpdfstring{ $\pair{\omega}{k}$}{(ω,k)}-categories}\label{sec:basics}
\subsection{Strict \texorpdfstring{$\pair{\omega}{k}$}{(ω,k)}-categories}\label{subsec:okcat}

\begin{paragr}
  Consider first the small category $\glob$ of globes, whose objects are
the integers $n\in\N$ and whose morphisms 
are generated by a double family
  \[
  \cosce{n},\cotge{n}: n\to n+1,
\]
where $n\in\N$,
subject to the ``coglobular'' equations
\begin{align*}
  \cosce{n+1}\circ\cosce{n} & = \cotge{n+1}\circ\cosce{n},\\
  \cosce{n+1}\circ\cotge{n} & = \cotge{n+1}\circ\cotge{n}.
\end{align*}
Thus, for any pair $\pair mn$ of integers, the homset $\homset{\glob}{m}{n}$ contains exactly two morphisms if $m<n$, only the identity if $m=n$ and none if $m>n$.

\end{paragr}

\begin{definition}\label{def:globular_set}
  A {\em globular set} $X$ is a presheaf on $\glob$, which amounts to a sequence
$(X_n)_{n\in\N}$ of sets, together with a double sequence of {\em source} and {\em target} maps
\[
  \begin{tikzcd}
   X_0 & X_1 \tdouble{l}{"\sce{0}"'}{"\tge{0}"}& \tdouble{l}{"\sce{1}"'}{"\tge{1}"}\cdots & X_n \tdouble{l}{"\sce{n-1}"'}{"\tge{n-1}"}& \tdouble{l}{"\sce{n}"'}{"\tge{n}"} X_{n+1}& \cdots \tdouble{l}{"\sce{n+1}"'}{"\tge{n+1}"},
  \end{tikzcd}
  \]
subject to the globular equations
\begin{align*}
  \sce{n}\circ\sce{n+1} & = \sce{n}\circ\tge{n+1},\\
  \tge{n}\circ\sce{n+1}& = \tge{n}\circ\tge{n+1}.
\end{align*}
\end{definition}
\begin{paragr}
  We denote by $\oGlob$ the category of globular sets and natural
  transformations. Let $X$ be a globular set and  $m$, $n$ be integers
  such that $0\leq m<n$. The images of the two morphisms of
  $\homset{\glob}{m}{n}$ by $X$ are the two maps 
\begin{align*}
  \sce{m}^n & = \sce{m}\circ\cdots\circ \sce{n-1}\ \text{and}\\
   \tge{m}^n & = \tge{m}\circ\cdots\circ \tge{n-1}.
\end{align*}
We often drop the subscript and write $\sce{m},\tge{m}:X_n\to X_m$ in this case. The elements of $X_n$ are called {\em cells of dimension $n$}  or {\em $n$-cells}. For any $n$-cell $x\in X_n$ and $m<n$, the $m$-cells $\sce m(x)$ and $\tge m(x)$ are respectively the {\em $m$-source} and {\em $m$-target} of $x$. Whenever $\sce m(y)=\tge m(x)$, we say that $x$, $y$ are {\em $m$-composable}.
If $n>0$, $x\in X_n$, $u=\sce{n-1}(x)$ and $v=\tge{n-1}(x)$, we shall write $x:u\to v$.
Finally, two $n$-cells $x,y\in X_n$ are {\em parallel} if either $n=0$, or $n>0$ and $x$, $y$ have the same $(n{-}1)$-source and the same $(n{-}1)$-target.

\end{paragr}

\begin{paragr}
  Let $n\in \N$. By discarding everything above level $n$ in the previous construction, we get the notions of {\em $n$-globular set} and corresponding category $\nGlob n$. For any $m\in\N\cup\set{\oo}$ such that $n<m$, there is an obvious truncation functor $\nGlob m\to \nGlob n$, right adjoint to the inclusion functor $\nGlob n\to \nGlob m$.
\end{paragr}

\begin{definition}\label{def:oo-cat}
 A strict \oo-category  is a globular set $C$ equipped with a family of identity maps
$\unit n:C_{n-1}\to C_n$ for $n>0$ and a family of partial binary composition operations  $\comp i:C_n\times C_n\to C_n$ for $0\leq i< n$, subject to the following list of conditions. 
\begin{enumerate}
\item
  \begin{enumerate}
  \item For each $n>0$ and $x\in C_{n-1}$, $\unit{n}(x): x\to x$.
  \item If $0\leq i<n$ and $x,y\in C_n$ are $i$-composable, then the composite $z=x\comp i y\in C_n$ is defined and
    \begin{enumerate}
    \item for $i<n-1$,
      \begin{align*}
        \sce{n-1}(z) & =\sce{n-1}(x)\comp i\sce{n-1}(y)\quad \hbox{and}\\
        \tge{n-1}(z) & =\tge{n-1}(x)\comp i\tge{n-1}(y) ;
      \end{align*}
    \item  for $i=n-1$,
      \begin{align*}
        \sce{n-1}(z) & =\sce{n-1}(x) \quad \hbox{and}\\
         \tge{n-1}(z) & =\tge{n-1}(y).
      \end{align*}
    \end{enumerate}
    \end{enumerate}
\item
  \begin{enumerate}
  \item If $0\leq i <n-1$ and $x$, $y$ are $i$-composable $(n{-}1)$-cells, then
    \begin{align*}
      \unit n(x\comp i y) & =\unit n(x)\comp i \unit n(y).
    \end{align*}
  \item If $0\leq i<n$ and $x$ is an $n$-cell, then
    \[\unit n(\sce i(x))\comp i x=x=x\comp i\unit n(\tge i(x))\]
    where, for any $i$-cell $u$, $\unit n(u)$ stands for $\unit n\circ\cdots\circ\unit{i+1}(u)$.
  \item Compositions are {\em associative}, that is, if $0\leq i<n$ and $x$, $y$, $z$ are $n$-cells such that $x$, $y$ and $y$, $z$ are $i$-composable, then
    \[(x\comp i y)\comp i z=x\comp i (y\comp i z).\]
  \item The {\em exchange rule} holds, that is, if $0\leq i<j<n$ and $x$, $y$, $z$, $t$ are $n$-cells such that $x$, $y$ and $z$, $t$ are $j$-composable, whereas $x$, $z$ and $y$, $t$ are $i$-composable, then
    \[(x\comp j y)\comp i (z\comp j t)= (x\comp i z)\comp j (y\comp i t).\]
  \end{enumerate}
\end{enumerate}
\end{definition}

\begin{paragr}
  Let $C$, $D$ be two \oo-categories. A {\em morphism} $f:C\to D$ is a
morphism of the underlying globular sets preserving identities and
$\comp i$-compositions. Strict \oo-categories and morphisms form a
category denoted by $\oCat$. Whenever necessary to avoid confusion, we use the notations $\sce i^C$ and
$\tge i^C$ to
emphasize that the corresponding source and target maps pertain to the
\oo-category $C$.
\end{paragr}

\subsubsection{Disks and spheres}\label{ssubsec:disks}
For each $n\in\N$, we denote by $\Dsk n$ the representable globular
set $\homset{\glob}{-}{n}$, having exactly two $m$-cells for $m<n$, one
$n$-cell, and no $m$-cell for $m>n$. By removing the only $n$-cell in
$\Dsk n$, one obtains a subrepresentable globular set which we denote by $\Sph{n-1}$. 
By abuse of notation, $\Dsk n$ and $\Sph{n-1}$ will
denote the corresponding freely generated  \oo-categories~(see~\cite[14.2]{abgmmm:polybk}). These come
with a family of inclusion morphisms $\gcof n:\Sph{n-1}\to\Dsk
n$. Note that $\Sph{-1}=\emptyset$ is the initial \oo-category.
\begin{remark}
  For any $n>0$ and $n$-cell $x$, the pair $\pair{\sce{n-1}(x)}{\tge{n-1}(x)}$ will be called the {\em type} of $x$. Thus the conditions of the first group above ensure that the units and composition operations are well-typed. The conditions of the second group are rules of computation and one easily checks that all equations preserve the types. 
\end{remark}
\subsubsection{\texorpdfstring{$n$}{n}-Categories}\label{ssubsec:ncat}
For a given $n\in\N$, truncating the previous construction at dimension $n$ yields the notion of $n$-category.  We denote by $\nCat n$ the resulting category. For example $\nCat 0$ is the category $\Set$ of sets and $\nCat 1$ is the category of small categories if we interpret $f\comp 0 g$ as the composition $g\circ f$ of two morphisms.
For $n\in\N$ and $m\in\N\cup \set{\oo}$ such that $n<m$, there are 
obvious truncation functors $\trk n:\nCat m\to \nCat n$ discarding all
cells of dimension $>n$. The functor $\trk n$ is right adjoint to the
inclusion functor  $\inc nm:\nCat n\to\nCat m$, which takes an
$n$-category $C$ to the $m$-category with the same $i$-cells for
$i\leq n$ and only the required identity cells in dimensions
$i>n$. The functor $\inc n{\oo}$ will be simply denoted by $\oinc
n$. Finally, for each $n\in\N$, the {\em $n$-skeleton functor} is by
definition the endofunctor 
\[
  \skl n=\oinc n\circ\trk n:\oCat\to \oCat.
  \]

\subsubsection{Invertible cells}\label{ssubsec:invert}
Let $C$ be an \oo-category, and $0\leq i <n$. An $n$-cell $x$ of $C$ is {\em $\comp i$-invertible} if there is an $n$-cell $y$ such that $\sce i(y)=\tge i(x)$, $\tge i(y)=\sce i(x)$ satisfying the equations 
  \begin{align*}
    x\comp i y & = \unit n(\sce i(x)) \quad \hbox{and}\\
    y\comp i x & =  \unit n(\tge i(x)).
  \end{align*}
  Such an $n$-cell $y$ is necessarily unique if it exists, in which case
  we say that $y$ is  the {\em $\comp i$-inverse} of $x$. A
  $\comp {n-1}$-invertible $n$-cell $x$ is simply called {\em
    invertible} and we denote its $\comp{n-1}$-inverse by $\inv x$. 
  \begin{proposition}\label{prop:invcell}
    If $0\leq i< j<n$ and $x$ is a $\comp i$-invertible cell, it is also $\comp j$-invertible.
  \end{proposition}
  \begin{proof}
    See~\cite[14.5]{abgmmm:polybk}. Beware that the $\comp j$-inverse
    of $x$ is not the same $n$-cell as its $\comp i$-inverse except in
    very degenerate cases. 
  \end{proof}
  \begin{definition}
    Let $k\in\N\cup\set{\oo}$. An {\em \ook k-category} $C$ is
    an~\oo-category whose $n$-cells are invertible for all $n>k$.
    \end{definition}
  \begin{paragr}\label{paragr:invert}
    It turns out that \ook k-categories satisfy the stronger
    condition of $\comp i$-invertibility for {\em all} $k<i<n$ (see
    again~\cite[14.5]{abgmmm:polybk}).  The full subcategory of
    $\oCat$ whose objects are \ook{k}-categories is denoted by $\opCat
    k$. Thus $\opCat{\oo}$ is just $\oCat$. Likewise, one defines the truncated version $\npCat nk$ for
    any $0\leq k<n$.  For example $\npCat 10$ is the category $\Grpd$
    of small groupoids and $\opCat 0$ the category $\oGrpd$ of
    \oo-groupoids. For any $0 \leq k \leq k'\leq \oo$, there is a canonical
    inclusion functor
    \[
      \emb{k}{k'} \colon \opCat{k} \to \opCat{k'}.
    \]
    This functor has a left adjoint
    \[
      \pifun{k}{k'} \colon \opCat{k'} \to \opCat{k}
    \]
    that formally adds inverses to all $n$-cells with $k \leq n < k'$,
    as well as a right adjoint 
    \[
      \maxgr{k}{k'}\colon\opCat{k'}\to \opCat{k}
    \]
    sending an \ook{k'}\nbd-category to its maximal
    sub\nbd-\ook{k}\nbd-category. For $k'=\oo$, we shall simply denote
    by $\oemb k$, $\opifun k$ and $\omaxgr k$ the corresponding
    functors. Observe that for each $k$, the category $\opCat k$ is
    preserved by the above defined $n$-skeleton functors. The
    corresponding restricted functors will still be denoted by $\skl n:\opCat
    k\to\opCat k$ in the following.
 \end{paragr}

 \subsection{\texorpdfstring{$\pair{\omega}{k}$}{(ω,k)}-polygraphs}\label{subsec:okpol}
 \subsubsection{Cellular extensions}\label{ssubsec:cellext}
 Let $n>0$, $C$ be an $n$-category, and $S$ be a set. A {\em cellular
   extension} of $C$ by $S$, denoted by $\clx CS$, is a pair of maps
 \[\sce n^S,\tge n^S:S\to C_n\]
 such that, for each $a\in S$, $\sce n^S(a)$ and $\tge n^S(a)$ are
 parallel $n$-cells. These maps are left implicit in the notation
 $\clx CS$, but
 are of course part of the data. 
 
Let $\clx CS$, $\clx DT$ be two cellular extensions of the
$n$-categories $C$ and $D$ respectively.  A {\em morphism}
 \[\phi:\clx CS\to\clx DT\]
 consists of a pair $\phi=\pair fh$ where $f:C\to D$ is a morphism in
 $\nCat n$ and $h:S\to T$ is a map such that, for each $a\in S$, $\sce
 n^T(h(a))=f_n(\sce n^S(a))$ and $\tge n^T(h(a))=f_n(\tge n^S(a))$.

 Cellular extensions and morphisms form a category $\nCatp
 n$. Likewise, when restricted to $\npCat nk$, the
 construction yields a category $\npCatp nk$. Now,
 any $(n{+}1)$-category $C$ determines a cellular extension $\clx{\trk
   n C}{S}$ of its truncated $n$-category $\trk n C$ by
 $S=C_{n+1}$, where $\sce n^S=\sce n^C$ and $\tge n^S=\tge n^C$.
 This gives a forgetful functor
 \[\fgf n:\nCat {n+1}\to \nCatp n.\]
 This functor $\fgf n$ admits a left adjoint
 \[\frf n:\nCatp n\to\nCat {n+1}.\]
Now, for any $k\leq n$, the functor $\fgf n$ restricts to
$\npCat{n+1}{k}$ and if $C$ is an~$\pair{n{+}1}{k}$-category, then,
for $k<n$, 
$\fgf nC$ is a cellular extension of an $\pair nk$-category. In the
special case where $k=n$, $\fgf nC$ is of course a cellular extension
of the (plain) $n$-category $\trk nC$.
Thus, for $k<n$, we get a forgetful functor
\[\fgf{n,k}:\npCat{n+1}{k}\to \npCatp nk.\]
This functor $\fgf{n,k}$ admits a left adjoint
\[\frf{n,k}:\npCatp nk\to \npCat{n+1}{k}.\]
Likewise, in the special case where $n=k$, we get a functor $\fgf
{n,n}:\npCat{n+1}{n}\to \nCatp n$
and its left adjoint $\frf{n,n}:\nCatp n\to \npCat{n+1}{n}$.
\cref{sec:syntax} below is devoted to a precise description of $\frf
{n,k}$ for a fixed $k$ and any $n\geq k$.
We also refer to~\cite{batanin:comfmg} and~\cite[Ch.18]{abgmmm:polybk} for
a comprehensive account over this adjunction. 

\subsubsection{Polygraphs}\label{ssubsec:polyg}
Let $k\in\N$ be a fixed integer. We define, by induction on $n\in\N$,
the notions of {\em $n$-polygraph} up to $n=k$ and {\em
  $\pair{n}{k}$-polygraph} from $n=k{+}1$ on.
\begin{itemize}
\item For $n=0$, an
$n$-polygraph $P$ is just a set $P_0$, and the $n$-category it
generates, denoted by $\frcat P$, is just the same set $\frcat P_0=P_0$.
\item Let $n<k$ and suppose we have defined the notion of
  $n$-polygraph together with the free $n$-category it
  generates. An $(n{+}1)$-polygraph $P$ is a cellular extension of the
  form $\clx{\frcat Q\!}{P_{n+1}}$ where $Q$ is an $n$-polygraph and
  $\frcat Q$ the free $n$-category generated by $Q$. The free
  $(n{+}1)$-category generated by $P$ is $\frcat P=\frf
  nP$.
\item For $n=k$, an~$\pair{n{+}1}{n}$-polygraph $P$ is a cellular
  extension of the form $\clx{\frcat Q\!}{P_{n+1}}$ where $Q$ is an
  $n$-polygraph, and the free $\pair{n{+}1}n$-category generated by
  $P$ is $\frgp P=\frf{n,n}P$.
  \item Let $n>k$ and suppose we have defined the notion of $\pair
    nk$-polygraph and freely generated $\pair nk$-category. An
    $\pair{n{+}1}{k}$-polygraph is a cellular extension  of the form
    $\clx{\frgp Q\!}{P_{n+1}}$ where $Q$ is an $\pair nk$-polygraph and
    $\frgp Q$ the free $\pair nk$-category generated by $Q$. The free
    $\pair{n{+}1}k$-category generated by $P$ is $\frgp P=\frf{n,k}
    P$.
 \end{itemize}   
The morphisms of $n$-polygraphs, for $n\leq k$ and~$\pair
nk$-polygraphs for $n>k$ are defined similarly by
induction on $n$. 
\begin{itemize}
\item If $P$, $P'$ are $0$-polygraphs, a morphism $p:P\to P'$ is
  just a map between the corresponding sets.
\item Let $n< k$ and $P$, $P'$ be two $(n{+}1)$-polygraphs defined by
  cellular extensions $\clx{\frcat Q\!}{P_{n+1}}$ and  $\clx{\frcat{Q'}\!}{P'_{n+1}}$ respectively. A morphism $p:P\to P'$ is a
  morphism
  \[\clx{\frcat q}{h}:\clx{\frcat Q\!}{P_{n+1}}\to \clx{\frcat
      {Q'}\!}{P'_{n+1}}\]
  of cellular extensions such that $q:Q\to Q'$ is a morphism of
  $n$-polygraphs and $\frcat q=\frf n q$.
\item Let $n= k$ and $P$, $P'$ be two $\pair{n{+}1}n$-polygraphs defined by
  cellular extensions $\clx{\frcat Q\!}{P_{n+1}}$ and  $\clx{\frcat
    {Q'}\!}{P'_{n+1}}$ respectively. A morphism $p:P\to P'$ is a
  morphism
  \[\clx{\frcat q}{h}:\clx{\frcat Q\!}{P_{n+1}}\to \clx{\frcat
      {Q'}\!}{P'_{n+1}}\]
  of cellular extensions such that $q:Q\to Q'$ is a morphism of
  $n$-polygraphs and $\frcat q=\frf n q$.  
  \item Let $n> k$ and $P$, $P'$ be two $\pair{n{+}1}k$-polygraphs defined by
  cellular extensions $\clx{\frgp Q\!}{P_{n+1}}$ and  $\clx{\frgp
    {Q'}\!}{P'_{n+1}}$ respectively. A morphism $p:P\to P'$ is a
  morphism
  \[\clx{\frgp q}{h}:\clx{\frgp Q\!}{P_{n+1}}\to \clx{\frgp
      {Q'}\!}{P'_{n+1}}\]
  of cellular extensions such that $q:Q\to Q'$ is a morphism of
  $\pair nk$-polygraphs and $\frgp q=\frf{n,k} q$.
\end{itemize}
Polygraphs and morphisms form a  sequence of categories
  \[
    \begin{tikzcd}[column sep=small]
    \nPol 0 &\nPol 1\ar[l]& \cdots\ar[l] & \nPol k\ar[l] & \npPol{k+1}k \ar[l]& \cdots\ar[l] &
    \npPol nk\ar[l] & \cdots\ar[l]  
    \end{tikzcd}
    \]
  where the arrows represent the obvious truncation functors. Finally,
  the category $\opPol k$ of \ook k-polygraphs is the projective limit of the
  above diagram of categories.
\begin{paragr}\label{paragr:okpol}
  For $n>k$, the structure of an $\pair nk$-polygraph $P$,  is
  summarized by the following diagram of sets:
  \[
  \begin{tikzcd}[row sep=huge]
    P_0\ar[d,"\gni 0" description]&P_1\ar[d,"\gni
    1" description]\tdouble{ld}{"\sce 0"'}{"\tge 0"}&\phantom{P_2}\tdouble{ld}{"\sce
      1"'}{"\tge 1"}&P_k\ar[d,"\gni{k}"
    description]\tdouble{ld}{"\sce{k-1}"'}{"\tge{k-1}"}&P_{k+1}\ar[d,"\gni{k+1}"
    description]\tdouble{ld}{"\sce k"'}{"\tge
      k"}&\phantom{P_{k+2}}\tdouble{ld}{"\sce{k+1}"'}{"\tge{k+1}"}&P_{n-1}\ar[d,"\gni{n-1}"
    description]\tdouble{ld}{"\sce{n-1}"'}{"\tge{n-1}"}&P_n\tdouble{ld}{"\sce n"'}{"\tge n"}\\
    \frcat P_0&\frcat P_1\tdouble{l}{""}{""}&\cdots\tdouble{l}{""}{""}&\frcat P_k\tdouble{l}{""}{""}&\frgp
    P_{k+1}\tdouble{l}{""}{""}&\cdots\tdouble{l}{""}{""}&\frgp P_{n-1}\tdouble{l}{""}{""}.&
  \end{tikzcd}
  \]
  At each dimension $l$, a new set $P_{l+1}$ of {\em generators} is
  attached, via the source and target maps $\sce l$, $\tge l$, to the free
  $l$-category for $l\leq k$, and  $\pair lk$-category for $l>k$ already built at this level.  The
  maps $\gni l:P_l\to \frcat{P}_l$ for $l\leq k$ and $\gni l:P_l\to
  \frgp{P}_l$ for $l>k$ represent the insertion of the generators in
  the $l$-cells of the freely generated category.
\end{paragr}
\begin{paragr}\label{paragr:opolpushout}
  Let $P$ be an $\ook k$-polygraph. The freely generated $\ook
  k$-category $\frgp P$ is the colimit of the following diagram of
  $\ook k$-categories
  \[
    \begin{tikzcd}
      0\ar[r]&\skl 0(\frgp P)\ar[r] & \cdots\ar[r] &\skl{n-1}(\frgp
      P)\ar[r] &\skl n(\frgp P)\ar[r]&\cdots
    \end{tikzcd}
    \]
  where $\skl{n-1}(\frgp P)\to\skl n(\frgp P)$ is
  obtained by taking the following pushout in $\opCat k$:
  \begin{equation}
    \label{eq:polpushout}
    \begin{tikzcd}
      \displaystyle\coprod_{P_{n}} \opifun k (\Sph{n-1})
     \ar[d,"\coprod_{P_{n}}\opifun k(\gcof {n})"',start anchor={[yshift=2.6ex]},end anchor={[yshift=-0.25ex]}]\ar[r]& \skl{n-1}(\frgp P\ar[d])\\
      \displaystyle\coprod_{P_{n}}\opifun k(\Dsk{n})\ar[r]& \skl
      n(\frgp P).
      \ar[from=1-1,to=2-2,phantom,"\ulcorner",very near end]
    \end{tikzcd}
  \end{equation}
  Here the top arrow assigns to each $n$\nbd-generator in $P_n$ a pair of parallel $(n{-}1)$-cells in
  $\skl{n-1}\frgp P$. This defines a cellular extension 
  $\clx{\trk{n-1}\frgp P}{P_n}$.  Now the fact that
  (\ref{eq:polpushout}) is a pushout square exactly amounts to the statement $\trk n\frgp
  P=\frf{n-1,k}\clx{\trk{n-1}\frgp P}{P_n}$.
\end{paragr}


%% file: syntax.tex
\section{Syntax}\label{sec:syntax}

\subsection{Well-typed formulas}\label{subsec:wtform}

\subsubsection{Formal expressions}\label{ssubsec:formex}
Let $0\leq k\leq n$, $C$ be an $\pair nk$-category, and  $\clx CS$ be a
cellular extension of $C$ by a set $S$, with $\sce n^S,\tge n^S:S\to C_n$. For each $a\in S$ and $x$,
$y\in C_n$, the notation $a:x\to y$ means 
$x=\sce n^S(a)$ and $y=\tge n^S(a)$. To $\clx CS$ we first associate
the formal language $\expr S$ whose formulas are given by the grammar
\[
  e ::= \fcst{a} \mid \finv{a} \mid \fid{x} \mid (e\fcomp i e)
\]
where $a\in S$, $x\in C_n$ and $0\leq i\leq n$. Note that this
language has the unique parsing property. This allows us to reason by
structural induction on the complexity of formulas in $\expr S$. Thus,
in order to completely determine a map $f$ from $\expr S$ to an arbitrary
set $X$, it suffices to define $f(e)$ for $e$ of the form $\fcst a$,
$\finv a$ or $\fid x$ and to provide the rules to evaluate $f((e_1\fcomp
i e_2))$ from the values $f(e_1)$ and $f(e_2)$ for each $i\leq n$.
For instance, we may define a {\em length} map $\lgh:\expr S\to \N$
by setting $\lgh(\fcst a)=\lgh(\finv a)= \lgh(\fid x)=1$ and $\lgh
(e_1\fcomp i e_2)=\lgh(e_1)+\lgh(e_2)$.

\subsubsection{Typing}\label{ssubsec:types}
We now define the
subset  $\wtx S\subset\expr S$ of {\em well-typed formulas}.
Precisely, we define the statement
\[\type wxy\]
meaning that {\em $w\in\wtx S$ is of type $\pair xy$} with $x$, $y$ parallel
cells in $C_n$, by the following structural induction:
\begin{itemize}
\item for each $a:x\to y\in S$ ,
  $\type{\fcst{a}}{x}{y}$
  and $\type{\finv{a}}{y}{x}$;
\item for  each $x\in C_n$, $\type{\fid{x}}{x}{x}$;
 \item if $\type wxy$ and $\type{w'}yz$, then $w''=(w\fcomp n w')$ is
   well-typed and $\type{w''}{x}{z}$;
 \item if $0\leq i<n$, $\type wxy$, $\type{w'}{x'}{y'}$ and $\sce
   i{x'}=\tge i{x}$, then $w''=(w\fcomp i w')$ is well-typed and
   $\type{w''}{x\comp i x'}{y\comp i y'}$. 
 \end{itemize}
 One readily checks that whenever $\type wxy$ is derivable from the
 preceding rules, then $x$, $y$ are well defined parallel cells in
 $C_n$, so that the maps $\sce n^S$, $\tge n^S$ extend to maps
 \[\scex n^S, \tgex n^S :\wtx S\to C_n\]
 such that, if $\type wxy$,  then $\scex n^S w=x$ and $\tgex n^S
 w=y$. The map $\finc S:S\to \wtx S$  taking
 $a\in S$ to $\fcst a\in\wtx S$ factorizes $\sce n^S$ and $\tge n^S$
 as $\sce n^S=\scex n^S \finc S$ and $\tge n^S=\tgex n^S\finc S$
 respectively. By composition with
 $\sce i^C,\tge i^C:C_n\to C_i$ for any $i<n$, we define likewise
 $\scex i^S,\tgex i^S:\wtx S\to C_i$. Note that the map $\finc S$ is injective. 
 
 \begin{paragr}\label{paragr:transl}
  Let $C$, $D$ be $\pair nk$-categories, $S$, $T$ be two sets,  $\clx
  CS$, $\clx DT$ be cellular extensions, and
  \[\phi=\pair fh: \clx CS\to\clx DT \]
  be a morphism in $\npCatp nk$. A map $g:\wtx S\to\wtx T$ is called a {\em
   $\phi$-compatible translation} if
  \begin{itemize}
  \item for each $w\in\wtx S$, $\scex n^Tg(w)=f_n(\scex n^S w)$ and
    $\tgex n^T g(w)=f_n(\tgex n^S w)$;
    \item for each $a\in S$, $g(\fcst a)=\fcst{h(a)}$ and $g(\finv
      a)=\finv{h(a)}$;
    \item for each $x\in C_n$,  $g(\fid x)=\fid{f_n(x)}$;
    \item  for each $w_1,w_2\in\wtx S$ and $i\leq n$ such that
      $w=(w_1\fcomp i w_2)$ is well-typed, $g(w)=(g(w_1)\fcomp i g(w_2))$.  
    \end{itemize}
\end{paragr}
\begin{lemma}\label{lemma:transl}
  Let $\phi=\pair fh:\clx CS\to \clx DT$ be a morphism of cellular
  extensions. There is a unique $\phi$-compatible translation $\wtx\phi:\wtx S\to \wtx T$.
\end{lemma}
\begin{proof}
  By structural induction on $w\in\wtx S$, there is a unique map
  \[g:\wtx S\to \expr T\]
  such that $g(\fcst a)=\fcst{h(a)}$, $g(\finv
      a)=\finv{h(a)}$, $g(\fid x)=\fid{f_n(x)}$ and \[g((w_1\fcomp i
      w_2))=(g(w_1)\fcomp i g(w_2)).\] For this map $g$, we easily
      show, by structural induction on $w$, that
      whenever $\type wxy$, then $\type{g(w)}{f_nx}{f_n y}$. Therefore
      $g$ factors through the inclusion $\wtx T\hookrightarrow \expr T$ of
      well-typed formulas and yields the required $\phi$-compatible
      translation \[\wtx\phi:\wtx S\to \wtx T.\qedhere\] 
    \end{proof}
    \begin{remark}\label{rmk:functoriality_of_wtx}
      By \cref{lemma:transl}, the correspondence
      \[
        \clx CS\mapsto \wtx S, (\phi:\clx CS\to \clx
      DT)\mapsto \wtx\phi
        \]
     defines a functor 
      $\WW:\npCatp nk\to\Set$.
    \end{remark}

 \begin{paragr}
   Let $\clx CS$ be a cellular extension of the $\pair nk$-category
   $C$ by a set $S$, $D$ be an $\pair{n{+}1}k$-category, and $f:C\to \trk n
   D$ be a morphism. A map $g:\wtx S\to D_{n+1}$ is called {\em $f$-compatible} if:
   \begin{itemize}
   \item for each $w\in\wtx S$, $\sce n^Dg(w)=f_n(\scex n^S w)$ and
     $\tge n^D g(w)=f_n(\tgex n^S w)$;
    \item for each $a\in S$, $g(\finv a)$ is the inverse of $g(\fcst
      a)$;
    \item for each $x\in C_n$, $g(\fid x)=\unit{n+1}(f_n(x))$;
    \item for each $w_1,w_2\in\wtx S$ and $i\leq n$ such that
      $w=(w_1\fcomp i w_2)$ is well-typed, $g(w)=g(w_1)\comp i g(w_2)$.  
    \end{itemize}
    Given $C$, $S$, $D$ and $f$ as above, we obtain the following result.
  \end{paragr}
  \begin{lemma}\label{lemma:mapext}
    Let $h:S\to D_{n+1}$ be a map such that $\sce n^Dh=f_n \sce
    n^S$ and $\tge n^Dh=f_{n} \tge n^S$. There is a unique
    $f$-compatible map
    $\extend h:\wtx S\to D_{n+1}$
    such that $\extend h\finc S=h$.
  \end{lemma}
  \begin{proof}
    Uniqueness immediately follows from $f$-compatibility. As for the
    existence, we define, for each $w:x\to y\in\wtx S$,
  $\extend h(w):f_n(x)\to f_n(y)\in D_{n+1}$, by structural induction
  on $w$:
  \begin{itemize}
  \item for $w=\fcst a$, where $a:x\to y\in S$, $\extend
    h(w)=h(a):f_n(x)\to f_n(y)$;
  \item for $w=\finv a$, where $a:x\to y\in S$, $\extend
    h(w)=\inv{h(a)}:f_n(y)\to f_n(x)$;
  \item for $w=\fid x$, where $x\in C_n$, $\extend
    h(w)=\unit{n+1}(f_n(x))$;
  \item if $w_1:x\to y$ and $w_2:y\to z$ are such that $\extend
    h(w_1):f_n(x)\to f_n(y)$ and $\extend h(w_2):f_n(y)\to f_n(z)$,
    then $w=(w_1\fcomp n w_2):x\to z$ and $d=\extend h(w_1)\comp
    n\extend h(w_2)$ is a well defined cell of $D_{n+1}$ such that
    $d:f_n(x)\to f_n(z)$ and we may set $\extend h(w)=d$;
  \item if $i<n$, $w_1:x\to y$, $w_2:z\to t$, $\sce i^Cz=\tge i^Cx $
    are such that $\extend h(w_1):f_n(x)\to f_n(y)$ and
    $\extend h(w_2):f_n(z)\to f_n(t)$, then
    $w:x\comp i z\to y\comp i t\in \wtx S$ and $\sce i^D\extend
    h(w_2)=\sce i^Df_n(z)=f_i(\sce i^Cz)=f_i(\tge i^Cx)=\tge
    i^Df_n(x)=\tge i^D\extend h(w_1)$, so that $d=\extend h(w_1)\comp
    i\extend h(w_2)$ is a well defined cell in $D_{n+1}$ with
    $d:f_n(x\comp i z)\to f_n(y\comp i t)$ and we may set $\extend h(w)=d$.
  \end{itemize}
  By construction, the map $\extend h$ so defined is $f$-compatible
  and $\extend h\finc S =h$.
  \end{proof}

\subsection{Formal description of \texorpdfstring{$\frf{n,k}$}{F(n,k)}}
\label{subsec:congruence}

 \subsubsection{Equivalence among formal expressions}\label{ssubsec:eqform}
 We now define the binary relation of ``elementary move'' among well
 typed expressions, denoted by $\frel$, as follows:
\begin{enumerate}[label=\textbf{(e\arabic*)}]
\item\label{ei} for any pair $x$, $x'$ of $i$-composable cells in $C_n$,
  $(\fid{x}\fcomp i\fid{x'})\frel\fid{x\comp i x'} $;
\item\label{eii} for each $a:x\to y$ in $S$, $(\fid{x}\fcomp{n}\fcst{a})\frel
  \fcst{a}$ and $(\fcst a\fcomp{n} \fid{y})\frel \fcst a$;
 \item\label{eiii} for each $a:x\to y$ in $S$, $(\fid{y}\fcomp{n}\finv{a})\frel
   \finv{a}$ and $(\finv a\fcomp{n} \fid{x})\frel \finv a$;
\item\label{eiv} for each $a:x\to y$ in $S$, $0\leq i< n$, $u=\unit n(\sce i
  x)$ and $v=\unit n(\tge i x)$, $(\fid{u}\fcomp{i}\fcst{a})\frel
  \fcst{a}$ and $(\fcst a\fcomp{i} \fid{v})\frel \fcst a$;
\item\label{ev} for each $a:x\to y$ in $S$, $0\leq i< n$, $u=\unit n(\sce i
  x)$ and $v=\unit n(\tge i x)$, $(\fid{u}\fcomp{i}\finv{a})\frel
  \finv{a}$ and $(\finv a\fcomp{i} \fid{v})\frel \finv a$;  
 \item\label{evi} for all expressions $w_1,w_2,w_3$ in $\wtx S$ such that $w_1,w_2$ and
   $w_2,w_3$ are $i$-composable, $((w_1\fcomp i w_2)\fcomp i w_3)\frel
   (w_1\fcomp i(w_2\fcomp i w_3))$;
 \item\label{evii} for all expressions $w_1,w_2,w_3,w_4$ in $\wtx S$ and $i<j$
   such that $w_1,w_2$ and $w_3,w_4$ are
   $j$-composable and $w_1,w_3$ are $i$-composable,
   \[
     ((w_1\fcomp i w_3)\fcomp j (w_2\fcomp i w_4))\frel ((w_1\fcomp j
     w_2)\fcomp i (w_3\fcomp j w_4));
   \]
  \item\label{eviii} for each $a:x\to y$ in $S$, $(\fcst a\fcomp n\finv a)\frel
    \fid{x}$ and $(\finv a\fcomp n\fcst a)\frel \fid y$. 
  \end{enumerate}
  Let now $\freq$ denote the congruence generated by
  $\frel$ on $\wtx S$, that is, the smallest equivalence relation containing $\frel$ and compatible with compositions. Precisely, if
  $w,w'\in\wtx S$, the relation  ``$w$ rewrites in one step
  to $w'$'',
  denoted by $w\cfrel w'$, is defined inductively as follows: $w\cfrel
  w'$ if either (i)~$w\frel w'$ or (ii)~$w=(x\fcomp i u)$ and $w'=(x\fcomp i u')$
  with $u\cfrel u'$ or (iii)~$w=(u\fcomp i x)$,
  and $w'=(u'\fcomp i x)$ with $u\cfrel u'$. The relation $\freq$ is
  then the smallest equivalence relation containing $\cfrel$. In other
  words, $w\freq w'$ if and only if there is a finite sequence
  $w_0,...,w_l$ where $l\geq 0$ of formulas in $\wtx S$ such that
  $w=w_0$, $w'=w_l$ and for each $i<l$, either $w_i\cfrel w_{i+1}$ or
  $w_{i+1}\cfrel w_i$.

  Let us denote by $\wtxeq S$
  the set $\wtx S/\!\!\freq$ of equivalence classes of $\wtx S$ under
  $\freq$ and by
  \[\cans S: \wtx S\to \wtxeq S\]
  the canonical surjection taking an expression $w$ to its
  equivalence class $\cans S(w)=\fcls w$. As the type of an
  expression is invariant under $\frel$, therefore also under $\freq$, 
 the source and target maps
  \[\scex n^S,\tgex n^S:\wtx S\to C_n\]
  factor through $\cans S$, and define source and target maps on $\wtxeq S$:
  \[\scecl n^S,\tgecl n^S:\wtxeq S\to C_n.\]
  More generally, for any $i\leq n$, we define the $i$-source and
  $i$-target maps
  \[\scecl i,\tgecl i:\wtxeq S\to C_i\]
  by composing $\sce n^S,\tge n^S:\wtxeq S\to C_n$ with the maps $\sce i^C,\tge
  i^C:C_n\to C_i$ already given. Now, for any $i$-composable formal expressions
  $w,w'\in \wtx S$ and $w''=(w\fcomp i w')$, the class $\cans S(w'')=\fcls{w''}$ only
  depends on $\fcls w$ and $\fcls{w'}$. Hence we may define 
 $\comp i$-compositions in $\wtxeq S$ by
 \[\fcls{w}\comp i\fcls{w'}=\fcls{(w\fcomp i w')}.\]
 Also, to any $x\in C_n$ we associate its $(n{+}1)$-unit in $\wtxeq
  S$ by
  \[\unit{n+1}(x)=\fcls{\fid x}.\]
  \begin{lemma}\label{lemma:inv}
    For each $\type wxy\in \wtx S$, there is a $\type{w'}yx\in \wtx S$
    such that $(w\fcomp n w')\freq\fid x$ and $(w'\fcomp n w)\freq
    \fid y$.
  \end{lemma}
  \begin{proof}
    We reason by induction on the complexity of $w$:
    \begin{itemize}
    \item if $w=\fid u$ for some $u\in C_n$, then $x=y=u$ and we set
      $w'=w$;
    \item if $w=\fcst a$ for some $a\in S$, we set $w'=\finv a$;
    \item if $w=\finv a$ for some $a\in S$, we set $w'=\fcst a$;
     \item if $w=(w_1\fcomp n w_2)$, then $\type{w_1}xz$ and
       $\type{w_2}zy$ and by induction we get $\type{w'_1}zx$ and
       $\type{w'_2}yz$ such that $(w_1\fcomp n w'_1)\freq \fid x$,
       $(w'_1\fcomp n w_1)\freq \fid z$, $(w_2\fcomp n w'_2)\freq \fid
       z$
       and $(w'_2\fcomp n w_2)\freq \fid y$. Thus $w'=(w'_2\fcomp n
       w'_1)$ is well-typed and satisfies $(w\fcomp n w')\freq \fid x$
       and $(w'\fcomp n w)\freq \fid y$;
       \item if $w=(w_1\fcomp i w_2)$ for some $i<n$, then
         $\type{w_1}{x_1}{y_1}$
         and $\type{w_2}{x_2}{y_2}$ with $x=x_1\comp i x_2$ and
         $y=y_1\comp i y_2$. By induction, we get
         $\type{w'_1}{y_1}{x_1}$ and $\type{w'_2}{y_2}{x_2}$ such that
         $(w_l\fcomp n w'_l)\freq \fid{x_l}$ and $(w'_l\fcomp n
         w_l)\freq \fid{y_l}$ for $l=1,2$. Thus $w'=(w'_1\fcomp i
         w'_2)$ is well-typed and satisfies $(w\fcomp n w')\freq \fid x$
       and $(w'\fcomp n w)\freq \fid y$.
    \end{itemize}
  \end{proof}
  Let us define the $(n{+}1)$-globular set $\extend C$ by
  $\trk n\extend C=C$ and \[\sce n,\tge n:\extend C_{n+1}=\wtxeq S\to C_n.\]
  The axioms for $\frel$ ensure that $\extend C$, endowed with
  the $\comp i$-compositions and units just defined is an
  $(n{+}1)$-category, and by~\cref{lemma:inv}, all its $(n{+}1)$-cells
  are invertible, so that $\extend C$ is
  indeed an
  $\pair{n{+}1}{k}$-category.

  We claim that $\extend C$ so defined is
  in fact, up to isomorphism, the expected $\pair{n{+}1}{k}$-category
  $\frf{n,k}\clx CS$ freely generated by the cellular extension $\clx CS$.
  This amounts to prove that $\extend C$ satisfies the following universal property.
  
\begin{proposition}\label{prop:univprop}
   Let $D$ be an  $\pair{n{+}1}{k}$-category, $f:C\to\trk n D$ be a
   morphism of $\pair nk$-categories,
   and  $h:S\to D_{n+1}$ be a map such that, for any $a\in S$,
   \[
     \sce n^Dh(a)=f_n\sce n^S (a)
     \quad\text{and}\quad \tge n^D h(a)=f_n\tge n^S (a).
     \]
   There is a unique
    morphism $\phi:\extend C\to D$ in $\npCat{n+1}{k}$ such that $\trk
    n\phi=f$ and for each $a\in S$, $\phi_{n+1}(\fcls{\fcst a})=h(a)$.
  \end{proposition}
  \begin{proof}
We first prove uniqueness. As $\phi$ is supposed to coincide with $f$ on
dimensions $\leq n$, we only need to prove uniqueness for its
$(n{+}1)$-component
\[\phi_{n+1}:\wtxeq S\to D_{n+1}.\]
 Thus, suppose
that $\phi$ satisfies the conditions of the proposition and let
\[\psi=\phi_{n+1} \cans{S}:\wtx S\to D_{n+1}.\]
   Because $\phi$ is a morphism, the  map $\psi$ is
   $f$-compatible, and by hypothesis \[\psi(\fcst
   a)=\phi_{n+1}\cans S(\fcst a)=h(a).\] Therefore, by the uniqueness
   part of
   \cref{lemma:mapext}, $\psi=\extend h$. Hence $\psi$ is uniquely determined.
As $\cans S$  is surjective,
  $\phi_{n+1}$ is also uniquely determined and so is $\phi$.

  As for the existence, by \cref{lemma:mapext}, there is an
  $f$-compatible map
  \[\psi:\wtx S\to D_{n+1}\]
  such that $\psi(\fcst a)=h(a)$ for each $a\in S$. 
  In each case~\ref{ei} to \ref{eviii} one checks that $w\frel w'$
  implies $\psi(w)=\psi(w')$. Therefore, if $w\freq w'$, then
  $\psi(w)=\psi(w')$, whence $\psi$ factors through $\cans S$, that
  is, there is a map $\phi_{n+1}:\wtxeq S\to D_{n+1}$ such that
  $\psi=\phi_{n+1}\cans S$. This map preserves source and target,
  identities and $\comp i$-compositions and, by definition, $\phi_{n+1}(\fcls
  {\fcst{a}})=h(a)$, so that it extends $f$ to the required morphism $\phi$. 
\end{proof}
\begin{remark}\label{rmk:functoriality_of_wtxeq}
  The universal property of~\cref{prop:univprop} allows us to make the correspondence
  \[\clx CS\mapsto \wtxeq S\]
  into a functor $\clWW:\npCatp
  nk\to\Set$. Precisely, this functor $\clWW$ is the composite
  \[
    \xymatrix{\npCatp nk\ar[r]^{\frf{n,k}} & \npCat{n+1}k \ar[r]^(.6){(-)_{n+1}}& \Set}
  \]
  where $(-)_{n+1}$ takes an $\pair{n{+}1}k$-category to the set of
  its $(n{+}1)$-cells.

  Let now $\phi=\pair fh:\clx CS\to \clx DT$
  be a morphism of cellular extensions, and define $h':S\to\wtxeq T$
  by $h'(a)=\fcls{\fcst{h(a)}}$. By~\cref{prop:univprop},
  \[\extend\phi=\frf{n,k}\phi:\frf{n,k}\clx CS\to\frf{n,k}\clx DT\] is
  the unique morphism in $\npCat{n+1}k$ such that $\trk
  n\extend\phi=f$ and $\extend{\phi}_{n+1}(\fcls{\fcst a})=h'(a)$ for
  each $a\in S$. Now
  $\extend{\phi}_{n+1}$ is precisely $\wtxeq{\phi}$, so that, for each $a\in
  S$, $\wtxeq{\phi}(\fcls{\fcst a})=\fcls{\fcst{h(a)}}$.
\end{remark}

\begin{paragr}\label{paragr:addproperties}
  Let $C$ be an $\pair{n{+}1}{k}$-category and $S$ be a set. Any map
  $p:S\to C_{n+1}$ immediately defines a cellular extension $\clx{\trk n C}{S}$ of
  $\trk n C$ by $S$, with $\sce n^S=\sce n^C\circ p$ and $\tge
  n^S=\tge n^C\circ p$. Thus
\[\pair{\Unit{\trk n C}}{p}:\clx{\trk n C}{S}\to\fgf{n,k}C\]
is a morphism of cellular extensions.
Let $\phi:\frf{n,k}\clx{\trk n C}{S}\to C$ be the transpose of $\pair{\Unit{\trk n C}}{p}$ under the adjunction.
By~\cref{prop:univprop}, $\extend{\trk n C}$ satisfies the universal property
of
$\frf{n,k}\clx{\trk n C}S$. In particular, it is canonically isomorphic to
$\frf{n,k}\clx{\trk n C}S$, and its set of $(n{+}1)$-cells is $\wtxeq S$.
 Let $\fcls
p=\phi_{n+1}:\wtxeq S\to C_{n+1}$ and $\extend p=\fcls p\cans S:\wtx
S\to C_{n+1}$. By~\cref{lemma:mapext}, the map $\extend p$ is the
unique $\Unit{\trk nC}$-compatible map such that $\extend p\finc
S=p$. Therefore, each $p:S\to C_{n+1}$ yields a factorization of the
form
\begin{equation}
  \label{eq:factor}
  \xymatrix{C_{n+1} & \wtxeq S \ar[l]^{\fcls{p}}& \wtx S\ar[l]^(.4){\cans{S}} \ar@/_1em/[ll]_(.3){\extend{p}}& S\ar[l]^(.3){\finc{S}}\ar@/_2em/[lll]_p}
\end{equation}
\end{paragr}

\begin{lemma}\label{lemma:main_diagram}
    Let $C$, $D$ be $\pair{n{+}1}{k}$-categories and $f:C\to D$ be a
    morphism. Let $S$, $T$ be two sets endowed with maps $p:S\to
    C_{n+1}$ and $q:T\to D_{n+1}$. Let $h:S\to T$ be a map such that
    $f_{n+1}p=qh$. Then
    \[ \phi=\pair{\trk n f}{h}:\clx{\trk n
      C}S\to\clx{\trk n D}{T}\]
is a morphism of cellular extensions
    and the following diagram commutes:
    \begin{equation}
      \label{eq:main_diagram}
      \begin{tikzcd}
       C_{n+1} \ar[d,"f_{n+1}"']& \wtxeq S\ar[l,"\fcls{p}"]\ar[d,"\wtxeq\phi"'] & \wtx S
      \ar[l,"\cans{S}"]\ar[d,"\wtx\phi"] & S\ar[l,"\finc{S}"]\ar[d,"h"]
      \ar[lll,bend right=15,"p"']\\
      D_{n+1} & \wtxeq T\ar[l,"\fcls{q}"'] & \wtx T
      \ar[l,"\cans{T}"'] & T.\ar[l,"\finc{T}"'] \ar[lll,bend left=15,"q"]
      \end{tikzcd}
     \end{equation}
   \end{lemma}
   \begin{proof}
     The outer square commutes by hypothesis, and the top and bottom
     factorizations are instances
     of~(\ref{eq:factor}). By~\cref{lemma:transl}, $\wtx\phi$ is a
     $\phi$-compatible translation and in particular, for each $a\in
     S$, $\wtx\phi(\fcst a)=\fcst{h(a)}$ so that the right square
     commutes.

     For the middle square, set $h'=\cans T\finc Th:S\to\wtxeq T$. Both maps
     $\wtxeq\phi\cans S$ and $\cans T\wtx\phi$ are
     $f$-compatible. Moreover, for each $a\in S$,
\[\wtxeq\phi\cans S\finc S(a)=\wtxeq\phi(\fcls{\fcst a})=\fcls{\fcst{h(a)}}=h'(a)\]
by~\cref{rmk:functoriality_of_wtxeq} and
\[
\cans
     T\wtx\phi\finc S(a)=\cans T\finc T h(a)=h'(a).
  \]
Thus, by the
     uniqueness part of~\cref{lemma:mapext}, $\wtxeq\phi\cans S=\cans
     T\wtx\phi$, so that the middle square commutes.

     For the left square, let
     $h''=qh$. Both maps $f_{n+1}\fcls p\cans S$ and $\fcls
     q\wtxeq\phi\cans S$ from $\wtx S$ to $D_{n+1}$ are $f$-compatible. Also
     \[
       f_{n+1}\fcls p\cans S\finc S=f_{n+1}p=qh=h''
     \]
      and
      \[
      \fcls q\wtxeq\phi\cans S\finc S=\fcls q\cans T\wtx\phi\finc
      S=\fcls q\cans T\finc T h=qh=h''.
      \]
      By the uniqueness part
     of~\cref{lemma:mapext},
     $f_{n+1}\fcls p\cans S=\fcls q\wtxeq\phi\cans S$ and as $\cans S$
     is surjective, $f_{n+1}\fcls p=\fcls q\wtxeq\phi$. Therefore the
     left square commutes and we are done.
   \end{proof}


%% file: conduche.tex
\section{Basis lifting theorem}\label{sec:basislift}

\subsection{Discrete Conduché fibrations}\label{subsec:conduche}
\begin{definition}\label{def:conduche}
   Let $0 \leq k\leq \oo$ and $C$, $D$ be \ook k-categories. A morphism $f:C\to D$ is a {\em discrete Conduché fibration} if it satisfies the following conditions:
  \begin{description}
  \item[\cndi] For $0\leq n<m$ and  $y\in D_n$,  $x\in C_m$ such that $f(x)=\unit
    m(y)$, there is a unique $z\in C_n$ such that $x=\unit m(z)$.
    \item[\cndii] For $0\leq i<n$, if $y_1$, $y_2$ are $i$-composable cells in
      $D_n$ and $x\in C_n$ is such that $f(x)=y_1\comp i y_2$, there is a unique pair $\pair{x_1}{x_2}$ of $i$-composable cells in $C_n$ such that $x=x_1\comp i x_2$, $f(x_1)=y_1$ and $f(x_2)=y_2$.
    \end{description}
  \end{definition}
 \begin{paragr}\label{paragr:properties_conduche}
 It can be proved that the first condition logically
    follows from the second one, but it will be more convenient to treat
    them as separate conditions in the forthcoming arguments. Note that, as~\cndi\ 
    and~\cndii\ can be expressed as right-lifting conditions, 
    discrete Conduché fibrations are stable under pullback. Of
    course~\cndi\ and~\cndii\ make sense in $\npCat nk$ for any
    $n\in\N$ and define the corresponding notion of discrete Conduché
    fibration in all $n$-truncations of $\opCat k$. 
    Remark that $f:C\to D$
    is a discrete Conduché fibration in $\npCat nk$ if and only if
    $\oinc n f:\oinc n C\to\oinc n D$ is one in $\opCat k$. In
    particular, for $n=k=1$, we recover the original notion of
    discrete Conduché fibration among small
    categories~\cite{conduche:exisad}.  
  \end{paragr}
  \begin{paragr} \label{paragr:lift_conduche}
In the case where $C$, $D$ are plain \oo-categories, that is if
$k=\omega$, we know from~\cite[Theorem 6.11]{guetta:poldcf} that if
$f:C\to D$ is a discrete Conduché fibration and $D$ is freely
generated by a polygraph, then so is $C$. We aim to prove that the
result still holds for any $k\in\N$. Now a key feature of a free \oo-category is that its generators are
     exactly the indecomposable cells, hence uniquely determined
     (see~\cite[Section 4, Proposition 8.3]{makkai:worcom}). This is
     no longer true for \ook k-categories when $k<\omega$, which makes the general case a
     priori significantly more complicated. However, the proof
     of~\cite[Theorem 6.11]{guetta:poldcf} does not rely on the
     uniqueness of the basis, and adapts to the general case.
  \end{paragr}

\subsection{Lifting properties of discrete Conduché fibrations}\label{subsec:lifting}
\begin{paragr}\label{paragr:generators}
 Let us fix an integer $k\in\N$.  Let $n\geq k$, $C$, $D$ be two $\pair{n{+}1}k$-categories, and $f:C\to D$ be a
 morphism. We suppose that there exists a set $T\subset D_{n+1}$ such
  that $D$ is isomorphic to $\frf{n,k}\clx{\trk nD}{T}$.
Let $q:T\to D_{n+1}$ be the inclusion of
  generators and define $S\subset C_{n+1}$ by
  \[
 S=f_{n+1}^{-1}(q(T))=\setof{a\in C_{n+1}}{f_{n+1}(a)\in q(T)}.
\]
We denote by $p:S\to C_{n+1}$ the inclusion map.
  Let $h:S\to T$ be the map taking $a\in S$ to the unique $b\in T$ such
  that $f_{n+1}p(a)=q(b)$. By definition, $f_{n+1}p=qh$. Therefore
  $C$, $D$, $f$, $p:S\to C_{n+1}$ and $q:T\to D_{n+1}$ satisfy the
  hypotheses of~\cref{lemma:main_diagram} and we get a commutative
  diagram of shape~(\ref{eq:main_diagram}). Moreover, as $D$ is
  isomorphic to $\frf{n,k}\clx{\trk nD}{T}$,
  the map $\fcls q:\wtxeq{T}\to D_{n+1}$ is a
  bijection.
  \begin{paragr}\label{paragr:hypotheses}
    From now on, we take $f:C\to D$, $S$ and $T$ as
    in~\ref{paragr:generators} and keep the notations
    of~\ref{paragr:addproperties}
    and~\cref{lemma:main_diagram}. Moreover, in all the forthcoming
    lemmas~\ref{lemma:surj} to~\ref{lemma:inj}, we suppose that $f$ is a discrete Conduché fibration.
  \end{paragr}
\end{paragr}
\begin{lemma}\label{lemma:surj}
 The map $\fcls p:\wtxeq S\to
 C_{n+1}$ is surjective.
\end{lemma}
\begin{proof}
  By hypothesis, $\fcls q$ is bijective and we know that $\cans T$ is
surjective. Therefore $\extend q=\fcls q\cans T$ is surjective.  Thus,
for each $d\in D_{n+1}$, we choose a formula  $v_d\in \wtx T$ of
{\em minimal length} such
that $\extend q(v_d)=d$. Note that by construction $\scex
n^Tv_d=\sce n^D d$ and $\tgex n^Tv_d=\tge n^D d$.
We now show how to associate to each $c\in C_{n+1}$ a formula
$u_c\in\wtx S$ such that $\scex n^S u_c=\sce
n^C c$, $\tgex n^S u_c=\tge n^C c$ and $\extend p(u_c)=c$. The
correspondence $c\mapsto u_c$ is defined by structural induction on
the complexity of $v_d$ where $d=f_{n+1}(c)$:
\begin{itemize}
\item if $v_d=\fid y$ for some $y\in D_n$, then $d=\extend q(\fid
  y)=\unit{n+1}(y)$ and by~\cndi\
  there is a unique $x\in C_n$ such that $c=\unit{n+1}(x)$ and $f_n(x)=y$, in
  which case we choose  $u_c=\fid x$, whence $\scex n^Su_c=x=\sce n^Cc$
  and $\tgex n^Su_c=x=\tge n^Cc$;   
\item if $v_d=\fcst b$ for $b\in T$, then $d=q(b)$ and as  $c\in
  f_{n+1}^{-1}(d)$, $c\in S$, so that $c=p(c)=\extend p(\fcst c)$ and we
  choose $u_c=\fcst c$, whence $\scex n^Su_c=\sce n^Cc$
  and $\tgex n^Su_c=\tge n^Cc$; 
 \item if $v_d=\finv b$ for $b\in T$, then $\inv d=\inv{\extend q(\finv b)}=\extend q(\fcst b)$, which
   implies $v_{\inv d}=\fcst b$ because $\fcst b$ is the unique
   formula of minimal length whose image by $\extend q$ is $\inv
   d$. Therefore $\inv d=q(b)\in q(T)$ and as $f_{n+1}(\inv
   c)=\inv{(f_{n+1}(c))}=\inv d$, we get $\inv c=a\in S$. We then
   choose $u_c=\finv a$, whence
   $\scex n^Su_c=\tge n^C\inv c=\sce n^Cc$,
   $\tgex n^Su_c=\sce n^C\inv c=\tge n^Cc$ and $\extend
   p(u_c)=\inv{(\inv c)}=c$;
 \item if $v_d$ is of the form $(v_1\fcomp i v_2)$ for $i\leq n$,
   $v_1,v_2\in \wtx T$, then $d_1=\extend q(v_1)$ and $d_2=\extend q
   (v_2)$ are $i$-composable and $d=d_1\comp i d_2$. By~\cndii, there is a unique pair $c_1$, $c_2$ of
   $i$-composable cells in $C_{n+1}$ such that $f_{n+1}(c_1)=d_1$,
   $f_{n+1}(c_2)=d_2$ and $c=c_1\comp i c_2$. 
By the minimal length condition $\lgh(v_{d_1})~\leq~\lgh(v_1) $ and $\lgh(v_{d_2})\leq \lgh(v_2)$, so that
   $\lgh(v_{d_1})<\lgh(v_d)$ and $\lgh(v_{d_2})<\lgh(v_d)$ and the
   induction hypothesis applies. Therefore we get
   $u_{c_1},u_{c_2}\in\wtx S$ such that $\extend p(u_{c_1})=c_1$ and
   $\extend p(u_{c_2})=c_2$. By induction hypothesis, the source and
   target of $u_{c_1}$ and $u_{c_2}$ are those of $c_1$ and $c_2$
   respectively. Therefore $\scex i^Su_{c_2}=\sce i^Cc_2=\tge
   i^Cc_1=\tgex i^Su_{c_1}$, so that the formula $(u_{c_1}\fcomp i
   u_{c_2})$ is well typed and we may define $u_c=(u_{c_1}\fcomp i
   u_{c_2})$. By definition, $\extend p(u_c)=c$. Finally we easily check
   that $\scex n^Su_c=\sce n^Cc$ and $\tgex n^Su_c=\tge n^C c$ in both
   cases $i=n$ and $i<n$. 
 \end{itemize}
 By constructing a section $c\mapsto u_c$ of the map $\extend p$ we
 have shown that $\extend p$ is surjective. As $\extend p=\fcls p\cans
 S$, this implies that $\fcls p$ is also surjective.
\end{proof}
\begin{paragr}\label{paragr:restricted_formulas}
  Let $C$, $S$, $p:S\to C_{n+1}$ be as in~\ref{paragr:generators}. For
  each $c\in C_{n+1}$ we define
  \[\wtxr cS=\setof{w\in\wtx S}{\extend p(w)=c}.\]
  Now, in the hypotheses of~\cref{lemma:main_diagram}, for each $c\in
  C_{n+1}$ and $d=f_{n+1}(c)$, if $w\in\wtxr cS$, then $\extend
  q\wtx\phi(w)=f_{n+1}\extend p(w)=f_{n+1}(c)=d$, so that
  $\wtx\phi(w)\in\wtxr dT$. We may therefore define
  \[\wtxr c\phi:\wtxr cS\to \wtxr dT\]
  as the restriction of $\wtx\phi$ to $\wtxr cS$. 
\end{paragr}
\begin{lemma}\label{lemma:injectivity}
  For all $c\in C_{n+1}$, the map
  $\wtxr c\phi$ is injective.
\end{lemma}
\begin{proof}
  Let $c\in C_{n+1}$ and $g=\wtxr c\phi$. Note that for each $w\in\wtxr
  cS$, $\lgh(g(w))=\lgh(w)$. Therefore it suffices to prove that for
  each $l>0$, the restriction of $g$ to the subset
  \[W^l_c=\setof{w\in\wtxr cS}{\lgh(w)=l}\]
  is injective. Let $\mathcal{P}_l$ denote the proposition:
  \begin{quote}
   ``For all $c\in C_{n+1}$, the restriction of $g$ to
    $W^l_c$ is injective''.
  \end{quote}
  We prove $\mathcal{P}_l$ by induction on $l>0$.
  \begin{enumerate}
  \item Suppose $l=1$. Let $c\in C_{n+1}$ and $w_1,w_2\in W^1_c$ be such that
    $g(w_1)=g(w_2)$. Three cases are possible:
    \begin{enumerate}
    \item if $w_1=\fcst a$ for some $a\in S$, then
      $g(w_1)=\fcst{h(a)}=g(w_2)$ so that $w_2$ is of the form
      $\fcst{a'}$ for some $a'\in S$, whence
      $g(w_2)=\fcst{h(a')}=\fcst{h(a)}$. Now $p:S\to C_{n+1}$ is by
      definition an inclusion map, hence injective, and
      $p(a)=p(a')=c$. Therefore $a=a'$ and $w_1=w_2$;
    \item if $w_1=\finv{a}$ for some $a\in S$, then $w_2$ is of the
      form $\finv{a'}$ for some $a'\in S$. Therefore
      $\finv{h(a)}=g(w_1)=g(w_2)=\finv{h(a')}$, which implies
      $h(a)=h(a')$ and also
      $g(\fcst{a})=\fcst{h(a)}=\fcst{h(a')}=g(\fcst{a'})$. By the
      previous case, this implies $a=a'$ and $w_1=w_2$;
      \item if $w_1=\fid x$ for some $x\in C_n$ then
      $g(w_1)=\fid{f_n(x)}=g(w_2)$, whence $w_2$ is of the form
      $\fid{x'}$ for some $x'\in C_n$ such that $f_n(x')=f_n(x)$. Now
      $\unit{n+1}(x)=\extend p(w_1)=c=\extend p(w_2)=\unit{n+1}(x')$
      so that $x=x'$ and $w_1=w_2$.
    \end{enumerate}
    \item Suppose $l>1$ and $\mathcal{P}_{l'}$ holds for all
      $l'<l$. Let $c\in C_{n+1}$ and $w_1,w_2\in W^l_c$ be such that
      $g(w_1)=g(w_2)$. Then $w_1$ is of the form $(w'_1\fcomp{i_1}
      w''_1)$ and  $w_2$ is of the form $(w'_2\fcomp{i_2} w''_2)$, but
      as $z=g(w_1)=g(w_2)$ is of the form $(z'\fcomp i z'')$,
      necessarily $i_1=i_2=i$,
      $g(w'_1)=g(w'_2)=z'$ and $g(w''_1)=g(w''_2)=z''$. Now
      $f_{n+1}(c)=\extend{q}(z')\comp i \extend{q}(z'')$ and by~\cndii\ there is a unique decomposition
      $c=c'\comp i c''$ such that $f_{n+1}(c')=\extend{q}(z')$ and
      $f_{n+1}(c'')=\extend{q}(z'')$. As $c=\extend{p}(w'_1)\comp
      i\extend{p}(w''_1)=\extend{p}(w'_2)\comp i\extend{p}(w''_2)$,
      $f_{n+1}\extend p(w'_1)=\extend qg(z')=\extend pg(w'_2)$ and
      $f_{n+1}\extend p(w''_1)=\extend qg(z'')=\extend pg(w''_2)$, the
      uniqueness of the decomposition implies $\extend p
      (w'_1)=\extend p(w'_2)=c'$ and  $\extend p(w''_1)=\extend p(w''_2)=c''$.
Therefore $w'_1,w'_2\in \wtxr {c'}S$ and $w''_1,w''_2\in \wtxr
{c''}S$. Because $\lgh(w'_1),\lgh(w'_2),\lgh(w''_1)$ and $\lgh(w''_2)$
are all $<l$, the induction hypothesis applies, whence 
      $w'_1=w'_2$ and $w''_1=w''_2$, so that $w_1=w_2$.
    \end{enumerate}
 Therefore  $\mathcal{P}_l$ holds for all $l>0$ and we are done.  
 \end{proof}
\begin{lemma}\label{lemma:forward_lift_frel}
 For each $w\in \wtx S$ and $z'\in\wtx
 T$ such that $\wtx{\phi}(w)\frel z'$ there exists $w'\in\wtx S$ such
 that $\wtx\phi(w')=z'$ and $w\frel w'$.
\end{lemma}
 \begin{proof}
   Let $g=\wtx\phi$, $w\in\wtx S$, $z=g(w)$ and $z'\in\wtx T$ be such that
   $z\frel z'$.

   Let us examine the different cases of~\ref{ssubsec:eqform}:
  In case~\ref{ei}, there are two $i$-composable cells $y_1,y_2\in D_n$ such that
  $z=(\fid{y_1}\fcomp i \fid{y_2})$ and $z'=\fid{y_1\comp i
    y_2}$. Then $w$ is of the form $(\fid{x_1}\fcomp i\fid{x_2})$ where
  $x_1$, $x_2$ are $i$-composable cells in $C_n$,
  $f_n(x_1)=y_1$ and $f_n(x_2)=y_2$. Let $w'=\fid{x_1\comp i
    x_2}$. We get $w\frel w'$ and $g(w')=z'$, whence the result.
  
In the first part of case~\ref{eii}  there is $b:y\to y'\in T$ such that $g(w)=(\fid y\fcomp n
\fcst b)$ and $z'=\fcst b$. Hence $w$ is of the form $(\fid x\fcomp n\fcst a)$ for
$a:x\to x'\in S$ such that $h(a)=b$. Then $w'=\fcst a$ satisfies
$g(w')=z'$ and $w\frel w'$, whence the result.
The second part of case~\ref{eii} is treated similarly, as well as
case~\ref{eiii}.

In the first part of case~\ref{eiv}, there is $b:y\to y'\in T$ such that
$g(w)=(\fid v\fcomp i \fcst b)$ with $v=\unit n(\sce i y)$ and
$z'=\fcst b$. Hence $w$ is of the form $(\fid u\fcomp i\fcst a)$ for
$a:x\to x'\in S$ such that $h(a)=b$, and $g(\fid u)=\fid{f_n(u)}=\fid
v$, which implies $f_n(u)=v=\unit n(\sce i y)$. By~\cndi, there is a unique
$u^i\in C_i$ such that $u=\unit n(u^i)$. Now $\fid u:u\to u$ and
$\fcst a:x\to x'$. As $w$
is well-typed, $\tge i u=\sce i x$. On the other hand, as
$u=\unit{n}(u^i)$, $\tge i u=u^i$. Therefore $w=(\fid u\fcomp i \fcst a)$
with $u=\unit n(\sce i x)$, so that $w'=\fcst a$ satisfies $w\frel
w'$. The second part of case~\ref{eiv} is similar, as well as case~\ref{ev}.

Cases~\ref{evi} and~\ref{evii} are straightforward.

In the first part of
   case~\ref{eviii}, there is a $b:y\to y'\in T$ such that $g(w)=(\fcst b\fcomp
   n\finv b)$ and $z'=\fid y$. Therefore $w$ is of the form $(\fcst{a_1}\fcomp 
   n\finv{a_2})$ for some $a_1:x_1\to x'_1\in S$,  $a_2:x_2\to x'_2\in
   S$ such that $x'_2=x'_1$, $h(a_1)=h(a_2)=b$ and
   $f_n(x_1)=f_n(x_2)=y$.  Now
   \begin{align*}
     f_{n+1}\extend p(w) & = f_{n+1}\extend p((\fcst{a_1}\fcomp n\finv{a_2}))\\
                                     & = f_{n+1}\extend p(\fcst{a_1})\comp n
                                       f_{n+1}\extend p(\finv{a_2})\\
                                     & = \extend q(\fcst b)\comp
                                       n\extend q(\finv b)\\
                                     & = \unit{n+1}(y).
   \end{align*}
   By  \cndi, there is a unique $x\in C_n$ such that
   $c=\extend p(w)=\unit{n+1}(x)$. Thus $\scex nw=\sce n\extend p(w)=x$
   and $\tgex nw=\tge n\extend p(w)=x$. On the other hand, $\scex
   nw=\scex n \fcst{a_1}=x_1$ and $\tgex nw=\tgex n\finv{a_2}=x_2$,
   whence $x_1=x_2=x$. If now $w_1=(\fcst{a_1}\fcomp n\finv{a_1})$,
   $\extend p(w_1)=\unit{n+1}(x)=\extend p(w)=c$, whence both $w$ and
   $w_1$
   belong to $\wtxr cS$ and as 
   $g(w_1)=g(w)$, by~\cref{lemma:injectivity}, $w=w_1$. Finally, by taking
   $w'=\fid x$, we get $g(w')=z'$ and $w\frel w'$, whence the result.
The second part of case~\ref{eviii} is similar.
 \end{proof}
 \begin{lemma}\label{lemma:backward_lift_frel}
 For each $w\in \wtx S$ and $z'\in\wtx
 T$ such that $z'\frel\wtx{\phi}(w)$ there exists $w'\in\wtx S$ such
 that $\wtx\phi(w')=z'$ and $w'\frel w$.
\end{lemma}
\begin{proof}
  Let $g=\wtx\phi$, $w\in\wtx S$, $z=g(w)$ and $z'\in\wtx T$ be such that
  $z'\frel z$.

  Let us examine the different cases of~\ref{ssubsec:eqform}: In
  case~\ref{ei}, there are two $i$-composable cells $y_1,y_2\in D_n$ such that
  $z=\fid{y_1\comp i y_2}$ and $z'=(\fid{y_1}\fcomp i \fid{y_2})$. As
  $g(w)=z$, there is $x\in C_n$ such that $f_n(x)=y_1\comp i y_2$ and
  $w=\fid{x}$. By \cndii, there is a unique pair
  $x_1,x_2\in C_n$ of $i$-composable cells such that $x=x_1\comp i
  x_2$, $f_n(x_1)=y_1$ and $f_n(x_2)=y_2$. Thus $w'=(\fid{x_1}\fcomp
  i\fid{x_2})$ satisfies $g(w')=z'$ and $w'\frel w$, whence the
  result.

  In the first part of case~\ref{eii}  there is $b:y\to y'\in T$ such that
  $z=\fcst b$ and $z'=(\fid y\fcomp n
\fcst b)$. Hence $w$ is of the form $\fcst a$ for
$a:x\to x'\in S$ such that $h(a)=b$. Thus $w'=(\fid{x}\fcomp n \fcst
a)$ satisfies $g(w')=z'$ and $w'\frel w$, whence the result.
The second part of case~\ref{eii} is treated similarly, as well as
cases~\ref{eiii} to~\ref{evii}.
In the first part of
   case~\ref{eviii}, there is a $b:y\to y'\in T$ such that $z=\fid{y}$ and $z'=(\fcst b\fcomp
   n\finv b)$. Therefore $w$ is of the form $\fid x$ for some $x\in
   C_n$ such that $f_n(x)=y$. Let $c=\extend p(w)$. As
   $f_{n+1}(c)=\extend q(z)=\extend q(z')=\extend q(\fcst b)\comp n
   \extend q(\finv b)$, by~\cndii, there is a unique
   decomposition $c=a\comp n a'$ such that $f_{n+1}(a)=\extend
   q(\fcst b)=q(b)$ and $f_{n+1}(a')=\extend q(\finv b)$. By
   definition, $a\in S$ so that $w'=(\fcst a\fcomp n\finv a)$ satisfies
   $g(w')=z'$ and $w'\frel w$, whence the result. The second part of
   case~\ref{eviii} is similar.
 \end{proof}

\begin{lemma}\label{lemma:forward_lift_cfrel}
 For each $w\in \wtx S$ and $z'\in\wtx
 T$ such that $\wtx{\phi}(w)\cfrel z'$ there exists $w'\in\wtx S$ such
 that $\wtx\phi(w')=z'$ and $w\cfrel w'$.
\end{lemma}
\begin{proof}
 Let $g=\wtx\phi$, $w\in\wtx S$, $z=g(w)$ and $z'\in\wtx T$ be such that
   $z\cfrel z'$. If~(i) $z\frel z'$, then
   by~\cref{lemma:forward_lift_frel}, there exists $w'\in\wtx S$ such
   that $g(w')=z'$ and $w\frel w'$, whence also $w\cfrel w'$. If~(ii)
   $z=(y\fcomp i v)$, $z'=(y\fcomp i v')$ with $v\cfrel v'$, then $w$
   is of the form $(x\fcomp i u)$ where $g(x)=y$ and $g(u)=v$. By
   induction, there exists $u'\in\wtx S$ such that $g(u')=v'$ and
   $u\cfrel u'$, so that $w'=(x\fcomp i u')$ satisfies $g(w')=z'$ and
   $w\cfrel w'$, as required. Case~(iii) where $z=(v\fcomp i y)$,
   $z'=(v'\fcomp i y)$ with $v\cfrel v'$ is similar to~(ii).
\end{proof}
\begin{lemma}\label{lemma:backward_lift_cfrel}
 For each $w\in \wtx S$ and $z'\in\wtx
 T$ such that $z'\cfrel\wtx{\phi}(w)$ there exists $w'\in\wtx S$ such
 that $\wtx\phi(w')=z'$ and $w'\cfrel w$.
\end{lemma}
\begin{proof}
  Same as~\cref{lemma:forward_lift_cfrel}, but using~\cref{lemma:backward_lift_frel} instead of~\cref{lemma:forward_lift_frel}.
\end{proof}

 \begin{lemma}\label{lemma:lift_freq}
 Let $c\in C_{n+1}$. For each $w,w'\in\wtxr c{S}$, if $\wtxr
  c\phi(w)\freq \wtxr c\phi (w')$, then $w\freq w'$.
\end{lemma}
\begin{proof}
  Let $g=\wtxr c\phi$ and $w,w'\in\wtxr cS$ be such that $g(w)\freq
  g(w')$. By~\ref{ssubsec:eqform}, there is a chain $z_0,...,z_l$ of
  formulas in $\wtx S$ with $l\geq 0$ such that $z_0=g(w)$, $z_l=g(w')$ and for each
  $i<l$, either $z_i\cfrel z_{i+1}$ or $z_{i+1}\cfrel z_i$. We reason
  by induction on the smallest such $l$. If $l=0$, then $g(w)=g(w')$,
  and by~\cref{lemma:injectivity} $w=w'$ whence $w\freq w'$. If
  $l>0$, either $z_0\cfrel z_1$ or $z_1\cfrel z_0$. In the first case,
  by~\cref{lemma:forward_lift_cfrel}, there is a $w_1\in\wtx S$ such
  that $g(w_1)=z_1$ and $w\cfrel w_1$, which implies $w_1\in\wtxr cS$,
  and by induction $w_1\freq w'$, whence also $w\freq w'$. In the
  second case, by~\cref{lemma:backward_lift_cfrel}, there is a
  $w_1\in\wtx S$ such that $g(w_1)=z_1$ and $w_1\cfrel w$, which
  implies $w_1\in\wtxr cS$ and by induction $w_1\freq w'$, whence also
  $w\freq w'$. 
\end{proof}
\begin{lemma}\label{lemma:inj}
 The map $\fcls p:\wtxeq S\to
 C_{n+1}$ is injective.
\end{lemma}
\begin{proof}
  Let $w,w'\in\wtx S$ be such that $\fcls p(\fcls w)=\fcls
  p(\fcls{w'})$, which implies $\extend p(w)=\extend p(w')=c$. Therefore
  \begin{align*}
    \extend q\wtxr c\phi(w) & = f_{n+1}\extend p(w)\\
                         & = f_{n+1}\extend p(w')\\
                         & = \extend q\wtxr c\phi(w').
  \end{align*}
  Now $\extend q=\fcls q\cans T$ and by hypothesis $\fcls q$ is a
  bijection, so that $\cans T\wtxr c\phi(w)=\cans T\wtxr c\phi(w')$, which
  means that $\wtxr c\phi(w)\freq \wtxr
  c\phi(w')$. By~\cref{lemma:lift_freq}, $w\freq w'$, whence $\fcls
  w=\fcls{w'}$ and we are done. \end{proof}
\begin{proposition}\label{prop:sum_up}
  Let $n\geq k$, $C$, $D$ be two $\pair{n{+}1}{k}$-categories, and $f:C\to D$ be a
  discrete Conduché fibration. Suppose there is a set $T\subset D_{n+1}$ such
  that $D$ is isomorphic to $\frf{n,k}\clx{\trk nD}{T}$ and  let
  $S=f_{n+1}^{-1}(T)\subset C_{n+1}$. Then $C$ is isomorphic to the
  $\pair{n{+}1}k$-category $\frf{n,k}\clx{\trk nC}{S}$.
\end{proposition}
\begin{proof}
  From Lemmas~\ref{lemma:surj} and~\ref{lemma:inj}, we deduce that
  $\fcls p:\wtxeq S\to C_{n+1}$ is a bijection. Therefore, up to
  isomorphism, the $\pair{n+1}k$-category $C$ is freely generated by the cellular
  extension $\clx{\trk nC}S$.
\end{proof}
\begin{remark}\label{remark:ncat}
 Given a cellular extension $\clx CS$ of an $\pair nk$-category $C$ by
 a set $S$ of $(n{+}1)$-generators, we have decribed the freely
 generated $\pair{n{+}1}{k}$-category $\frf{n,k}\clx CS$ by means of
 formal expressions $\wtx S$ and their equivalence classes $\wtxeq
 S$. Likewise, starting with a plain $n$-category $C$ and a cellular
 extension $\clx CS$ of it, the freely generated $(n{+}1)$-category
 $\frf n\clx CS$  is decribed by means of formal expressions $\wtxo S$
 and their equivalence classes $\wtxeqo S$. We simply drop from our
 language all  constants $\finv a$ as well as the
 axioms~\ref{ssubsec:eqform} \ref{eiii}, \ref{ev} and \ref{eviii} 
 pertaining to these formal inverses, and define relations $\frel$,
 $\cfrel$ and $\freq$ accordingly. Clearly all lemmas~\ref{lemma:surj}
 to~\ref{lemma:inj} still hold in this restricted setting
 and~\cref{prop:sum_up} immediately admits the restricted
 version below. Note that this result is already proved
 in~\cite{guetta:poldcf} by using a very similar if not exactly
 identical formal system.
\end{remark}
\begin{proposition}\label{prop:sum_up_restr}
  Let $n\in\N$, $C$, $D$ be two $(n{+}1)$-categories, and $f:C\to D$ be a
  discrete Conduché fibration. Suppose there is a set $T\subset D_{n+1}$ such
  that $D$ is isomorphic to $\frf n\clx{\trk nD}{T}$ and  let
  $S=f_{n+1}^{-1}(T)\subset C_{n+1}$. Then $C$ is isomorphic to the
  $(n{+}1)$-category $\frf n\clx{\trk nC}{S}$.
\end{proposition}

\begin{theorem}\label{thm:main}
  Let $0 \leq k\leq \oo$ and $C$, $D$ be two \ook k-categories. If $f:C\to
  D$ is a discrete Conduché fibration and $D$ is freely generated by an
  \ook k-polygraph $Q$, then
$C$ is also freely generated by an \ook k-polygraph $P$. The
$n$-generators of $P$ can be chosen as those cells $a\in C_n$ such that
$f_n(a)=q_n(b)$ for $b\in Q_n$, where $q_n:Q_n\to D_n$ is the inclusion map
of $n$-generators in $D_n$.
\end{theorem}
\begin{proof}
  Let $Q$ be an \ook{k}-polygraph such that
  $D=\frgp Q$. For each $n\geq 0$, let $q_n:Q_n\to D_n=\frgp Q_n$ be the
  inclusion map of generators. Let us define $P_n\subset C_n$ by
  \[
    P_n=\setof{a\in C_n}{f_n(a)\in q_n (Q_n)}
  \]
  and let $p_n:P_n\to C_n$ be the inclusion map. The map
  $h_n:P_n\to Q_n$ taking each $a\in P_n$ to the unique $b\in Q_n$
  such that $q_n(b)=f_n(a)$ fits into the commutative diagram
  \[
    \begin{tikzcd}
      C_n \ar[d,"f_n"'] & P_n \ar[l,"p_n"'] \ar[d,"h_n"]\\
      D_n & Q_n. \ar[l,"q_n"]
    \end{tikzcd}
\]
  Thus, for each
  $n\geq 0$, we get a cellular extension $\clx{\trk nC}{P_{n+1}}$ of
  $\trk n C$ by the subset $P_{n+1}\subset C_{n+1}$.

  Now, if $n<k$,
  by~\cref{prop:sum_up_restr}, $\trk{n+1}C$ is isomorphic to $\frf
  n\clx{\trk nC}{P_{n+1}}$, whereas if $n\geq k$, by
\cref{prop:sum_up},
  $\trk{n+1}C$ is isomorphic to $\frf
  {n,k}\clx{\trk nC}{P_{n+1}}$. Starting with $P_0=C_0$, we have built
  inductively an \ook k-polygraph
  $P$ such that $C$ is isomorphic to $\frgp P$. Therefore $C$ is
  freely generated by an \ook k-polygraph, whose generators can be
  chosen as required.
\end{proof}


%% file: resolproj.tex
\section{The polygraphic resolution theorem}\label{sec:projresol}

\subsection{Folk model structure}
\begin{paragr}\label{paragr:folk}
  Recall that there exists a so-called ``folk'' model category
  structure on the category $\oCat$ \cite{lafontetal:folkms}, characterized by the fact that
  all objects are fibrant and the cofibrations are generated by the
  set 
  \[
    \setof{\gcof{n}:\Sph {n-1} \to \Dsk n }{n \in \N}.
  \]
  The cofibrant objects of this model structure are exactly the
  \oo\nbd-categories free on a polygraph. For a description of the
  weak equivalences, see \cite[Theorem 6.19]{lafontetal:folkms}. 
\end{paragr}
\begin{paragr}
  More generally, for every $k \in \N$, there exists a model
  category structure on the category $\opCat{k}$ \cite{ara_lucas:folmon} which is
  right-induced from the inclusion functor
  \[
    \oemb k:\opCat{k} \to \oCat.
  \]
  In other words, every object of $\opCat{k}$ is fibrant and the cofibrations of $\opCat{k}$ are generated
  by the set 
  \[
    \setof{\opifun{k}(\gcof{n}):\opifun{k}(\Sph {n-1}) \to \opifun{k}(\Dsk n) }{n \in \N},
  \]
  where $\opifun{k} \colon \oCat \to \opCat{k}$ is left adjoint to
  $\oemb k$ (see~\cref{paragr:invert}).
\end{paragr}
\begin{definition}\label{def:respol}
  Let $0 \leq k\leq \oo$ and $C$ be an $\ook k$-category. A {\em
  polygraphic resolution} of $C$ in $\opCat k$ consists of an
$\ook k$-polygraph $P$ together with a trivial fibration
$p:\frgp P\to C$
in the folk model structure of $\opCat k$.
\end{definition}
\begin{paragr}\label{paragr:projcof}
  As with any cofibrantly generated model structure, given a small
  category $C$, we can consider the \emph{projective} model structure
  on the category of functors
  \[
    \funopCat{k}{C},
  \]
  for any $0 \leq k \leq \oo$. The weak equivalences, fibrations
  and trivial fibrations are the pointwise ones (in particular, every
  object is fibrant) and the cofibrations are generated by the set of
  morphisms
  \[
   \coprod_{\homset{C}{c}{-}}\opifun{k}(\gcof{n}): \coprod_{\homset{C}{c}{-}}\opifun{k}(\Sph{n-1}) \to \coprod_{\homset{C}{c}{-}}\opifun{k}(\Dsk{n}),
  \]
  for all objects $c$ of $C$ and $n \in \N$ (see~\cite[Theorem 11.6.1]{hirschhorn:modctl}).
\end{paragr}
\subsection{Slice categories}\label{subsection:slice_cat}
\begin{paragr}\label{paragr:slice_cat}
  Let $C$ be a small category. Recall that each object $c$ of $C$
  determines the slice category $\tr Cc$ whose objects are pairs
  $\pair{x}{f\colon x \to c}$ where $x$ is an object of $C$ and $f:x\to c$ is
  a morphism in $C$. A morphism $\pair xf\to \pair{x'}{f'}$ is a
  morphism $g:x\to x'$ in $C$ such that $f'g=f$. The correspondence
  $c\mapsto \tr Cc$ is then the object part of a functor
\[
    \begin{aligned}
      \tr{C}{-} \colon C &\to \Cat \\
      c &\mapsto \tr{C}{c}.
    \end{aligned}
  \]
  For each object $c$ of $C$, there is a canonical forgetful functor
  $\tr Cc\to C$
  taking $\pair x{f:x\to c}$ to $x$.
\end{paragr}

\begin{lemma}\label{lemma:slicescond}
  For each small category $C$ and each object $c$ of $C$, the canonical functor
 $\tr Cc\to C$
  is a discrete Conduché fibration.
\end{lemma}
\begin{proof}
Consider a morphism in $\tr Cc$ from $\pair xf$ to $\pair{x'}{f'}$, that
is, a morphism $g:x\to x'$ in $C$ such that the following triangle commutes:
\[
    \begin{tikzcd}[column sep=small]
      x \ar[rr,"g"] \ar[dr,"f"'] && x'\ar[dl,"f'"] \\
      &c&
    \end{tikzcd}.
  \]
  If $g=g_2\circ g_1$, then we have the following decomposition in
  $\tr{C}{c}$,
  \[
      \begin{tikzcd}
      x \ar[rr,"g"] \ar[dr,"f"'] && x'\ar[dl,"f'"] \\
      &c&
    \end{tikzcd}
    =
    \begin{tikzcd}
      x \ar[r,"g_1"] \ar[dr,"f"']& y \ar[r,"g_2"] \ar[d,"h"] &
      x' \ar[dl,"f'"] \\
      &c&
    \end{tikzcd},
  \]
  with $h=f'\circ g_2$. This choice of $h$ is clearly unique, which
  proves that $\tr{C}{c} \to C$ is a discrete Conduché fibration.
\end{proof}

\subsection{Polygraphic resolutions of small categories}\label{subsec:polres}

From now on, we shall identify any small category with its image
in $\oCat$ by the inclusion functor \[\oinc 1\colon\nCat 1\to \oCat\]
defined in \ref{ssubsec:ncat}. Note that, for any $k\geq 1$, this
image belongs to $\opCat k$. 
\begin{paragr}\label{paragr:pullback_resol}
  Let $C$ be a small category, $1\leq k \leq \oo$, $X$ be an $\ook
  k$\nbd-category, and $p \colon X \to C$ be a morphism of $\ook k$\nbd-categories. For an object $c$ of $C$, we define the $\ook k$\nbd-category $\tr{X}{c}$ as the following pullback in $\opCat{k}$
  \begin{equation}
  \label{eq:pullback}
  \begin{tikzcd}
      \tr{X}{c} \ar[d] \ar[r] & \tr{C}{c} \ar[d] \\
       X \ar[r,"p"'] & C.
      \ar[from=1-1,to=2-2,phantom,"\lrcorner",very near start]
    \end{tikzcd}
  \end{equation} 
  Let us describe the $n$-cells of $\tr{X}{c}$
  for all $n\geq 0$:
  \begin{itemize}
  \item a $0$-cell of $\tr{X}{c}$ is a pair $\pair{x_0}{f}$,
    where $x_0\in X_0$, $f\colon p_0x_0\to c\in (\tr Cc)_0$;
  \item for $n>0$, an $n$-cell of $\tr{X}{c}$ amounts to a
    pair $\pair{x_n}{f}$ where  $x_n\in X_n$ and $f:\tge 0
    (p_nx_n)\to c$.
  \end{itemize}
The construction of $\tr{X}{c}$ is functorial in $c$, thus defining a functor
  \[
    \begin{aligned}
      \tr{X}{-} \colon C &\to \opCat{k} \\
      c &\mapsto \tr{X}{c}.
    \end{aligned}
  \]
\end{paragr}
In the following theorem, we apply this construction in the case
$X=\frgp P$ for an $\ook k$-polygraph $P$ and $p \colon \frgp P \to C$
is a polygraphic resolution in $\opCat k$.
\begin{theorem}\label{thm:projresol}
    Let $C$ be a small category and $p:\frgp P\to C$ be a polygraphic
    resolution of $C$ in $\opCat k$, with $1\leq k\leq\oo$. The functor
    \[\tr{\frgp P\!}{-}:C\to \opCat k\]
    is a cofibrant object of the category $\funopCat{k}{C}$, equipped with the projective model structure.
\end{theorem}
  \begin{proof}
   Let $c$ be an object of $C$ and consider the
   square~(\ref{eq:pullback}). By~\cref{lemma:slicescond}, the
   morphism $\tr Cc\to C$ is a discrete Conduché fibration. By stability of the class of discrete Conduché fibrations under pullback, the morphism $\tr{\frgp P\!}c\to \frgp
   P$ is also a discrete Conduché fibration. As $P$ is an $\ook
   k$-polygraph, it follows from \cref{thm:main} and the description
   of the cells of $\tr{\frgp P\!}c$
   that $\tr{\frgp P\!}c$ is itself freely generated by an $\ook k$\nbd-polygraph $\tr Pc$ whose set of $0$-generators is
   \[
     (\tr Pc)_0=\setof{\pair{a}{f}}{a\in P_0\ \hbox{and}\  f:p_0\gni 0 a\to c},
   \]
   and whose set of $n$-generators for $n>0$ is
   \[
     (\tr Pc)_n=\setof{\pair{a}{f}}{a\in P_n\ \hbox{and}\  f:\tge 0(p_n\gni n a)\to c},
   \]
   where $\gni n$ inserts the $n$-generators in the $n$-cells of
   $\frgp P$ (see~\ref{paragr:okpol}). Thus, we have a canonical isomorphism
   $\tr{\frgp{P}\!}c\cong \frgp{(\tr Pc)}$.

   Using this isomorphism, the pushout square~(\ref{eq:polpushout})
   of~\ref{paragr:opolpushout} applies to the \ook k-polygraph $\tr
   Pc$ and yields a pushout
  \[
    \begin{tikzcd}[row sep=large]
           \displaystyle\coprod_{{(\tr Pc)}_{n}} \opifun k (\Sph{n-1})
     \ar[d,"\coprod_{{(\tr Pc)}_{n}}\opifun
     k(\gcof{n})"',start anchor={[yshift=2.8ex]}]\ar[r]&\skl{n-1}(\tr{\frgp{P}}{c})\ar[d]\\
           \displaystyle\coprod_{{(\tr Pc)}_{n}}\opifun k(\Dsk{n})\ar[r]& \skl n{(\tr{\frgp{P}}{c})}
           \ar[from=1-1,to=2-2,phantom,"\ulcorner",very near end]
         \end{tikzcd}
  \]
  in $\opCat k$ for all $n\geq 1$. All the morphisms in the above
  square are natural in $c$, so that we get corresponding pushout
  squares in~$\funopCat{k}{C}$:
    \[
    \begin{tikzcd}[row sep=large]
           \displaystyle\coprod_{{(\tr P-)}_{n}} \opifun k (\Sph{n-1})
     \ar[d,"\coprod_{{(\tr P-)}_{n}}\opifun
     k(\gcof{n})"',start anchor={[yshift=2.8ex]}]\ar[r]&\skl{n-1}(\tr{\frgp{P}}{-})\ar[d]\\
           \displaystyle\coprod_{{(\tr P-)}_{n}}\opifun k(\Dsk{n})\ar[r]& \skl n{(\tr{\frgp{P}}{-})}.
           \ar[from=1-1,to=2-2,phantom,"\ulcorner",very near end]
         \end{tikzcd}
  \]
      Using the above description of the generators of $\tr P-$, the left
  map decomposes as
  \[
    \coprod_{{(\tr P-)}_{n}}\opifun k(\gcof {n})=\coprod_{a\in
      P_n}\coprod_{f\in\homset{C}{\tge 0(p_n\gni na)}{-}}\opifun k(\gcof {n})
  \]
  which is a coproduct of generating cofibrations
  in~$\funopCat{k}{C}$ by~\ref{paragr:projcof}, hence a cofibration. As cofibrations are
  stable by pushout, the maps
  \[
    \skl{n-1}(\tr{\frgp{P}\!}{-})\to \skl n(\tr{\frgp{P}\!}{-})
\]
   are cofibrations. Finally, the functor
  $\tr{\frgp{P}\!}{-}$ is the colimit of the sequence of cofibrations
  \[
    \begin{tikzcd}[column sep=small]
      0\ar[r] &\skl 0(\tr{\frgp{P}\!}{-})\ar[r] &\cdots\ar[r]
      &\skl{n-1}(\tr{\frgp{P}\!}{-})\ar[r]&
      \skl n(\tr{\frgp{P}\!}{-})\ar[r] &\cdots
    \end{tikzcd}
  \]
  and therefore a cofibrant object in the projective model
  structure. Here $0$ denotes the constant functor taking each object
  of $C$ to the empty \ook k-category, which is the initial object of
  $\funopCat{k}{C}$.
\end{proof}
\subsubsection{Groupoids}\label{ssubsec:resol_grp}
Let $G$ be a groupoid, that is, a small category where all arrows are
isomorphisms. For each
object $c$ of $G$, the slice category $\tr Gc$ is still a groupoid. Consequently,
all the results of~\cref{subsection:slice_cat} hold when restricted to the
subcategory $\Grpd=\npCat 10$ of groupoids. Moreover $G$ can be
considered as an object of $\opCat 0$, and more generally as an object of $\opCat k$ for any $0\leq k \leq \oo$, via the canonical inclusion $\Grpd \to \opCat k$. We may therefore
define polygraphic resolutions of $G$ in $\opCat k$ as pairs $\pair
Pp$ where $P$ is an $\ook
k$-polygraph and  $p:\frgp P\to G$ a trivial fibration
in the folk model structure on $\opCat k$ for any $k\geq 0$.  Thus, we may state a
groupoidal version of~\cref{thm:projresol}.
\begin{theorem}\label{thm:projresol_grp} Let $G$ be a groupoid and
$p:\frgp P\to G$ be a polygraphic resolution of $G$ in $\opCat k$, with
$0\leq k\leq \oo$. The functor
    \[\tr{\frgp P\!}{-}:G\to \opCat k\] is a cofibrant object of the
category $\funopCat{k}{G}$, equipped with the projective model
structure.
\end{theorem}
\begin{proof} The proof is the same as that of~\cref{thm:projresol},
  with $k=0$ allowed.
\end{proof}
\begin{remark} Note that the proofs of Theorems~\ref{thm:projresol}
  and~\ref{thm:projresol_grp} do not use the acyclicity of $p$, only the
  fact that $\frgp P$ is freely generated, together with the properties
  of the slice projection established before. However the statements are
  interesting only in the case where $\frgp P$ is a cofibrant
  replacement of $G$.
\end{remark}
  

%% file: polhom.tex
\section{Polygraphic homology}\label{sec:polhom}
\subsection{Abelianization of \texorpdfstring{$\ook k$}{(ω,k)}-categories}\label{subsec:abel}
\begin{paragr}
  The {\em abelianization functor}
  \[\abel\colon\oCat \to \Chp \]
  from the category of $\oo$\nbd-categories to
  the category of chain complexes in non-negative degree is defined as follows. For an $\oo$\nbd-category $C$ and $n\geq 0$
  \[
    \abel_n(C):=\Z C_n/\sim_n,
  \]
  where $\Z C_n$ is the free abelian group on the set of $n$\nbd-cells
  of $C$ and $\sim_n$ is the smallest equivalence relation compatible with
  the abelian group structure such that $x\comp{i} y
  \sim _n x + y$ for any $x,y\in C_n$ that are
  $\comp{i}$-composable. For $x\in C_n$, we denote by $[x]$ its
  equivalence class in $\abel_n(C)$. The boundary operator is defined by
  \[
    \begin{aligned}
      \partial \colon \abel_{n+1}(C) &\to \abel_{n}(C)\\
      [x] &\mapsto [\target x] - [\source x].
    \end{aligned}
  \]
  One checks that this is well-defined, that $\partial
  \circ \partial=0$, and that the correspondence $C \mapsto
  \abel(C)$ is functorial in the evident
  way. Furthermore, $\abel$ admits a right adjoint. For
  details, we refer to \cite[22.1]{abgmmm:polybk}.

  By pre-composing with the inclusion functor, for any $k\geq 0$, we
  also get an abelianization functor
  \[
    \abelk{k} \colon \opCat{k} \overset{\oemb{k}}{\longrightarrow} \oCat \overset{\abel}{\longrightarrow} \Chp.
  \]
  In particular, for $k=\oo$, we have $\abel=\abelk{\oo}$ by
  definition. Notice that since $\oemb{k}$ admits a right adjoint, so
  does $\abelk{k}$.
  If $C$ is freely generated by an \ook k-polygraph, the complex
  $\abelk k(C)$ admits a simple, explicit description, based on the following
lemma.
\end{paragr}

\begin{lemma}\label{lemma:delooping}
  Let $P$ be an $\ook k$\nbd-polygraph and $C=\frgp P$ be the free $\ook
  k$\nbd-category on this polygraph. Let $n > 0$, $G$ be an
  abelian group and 
  \[
    f \colon P_n \to G,
  \]
  be a map.
  Then there is a unique map
  \[\extend f\colon C_n\to G\]
 satisfying the following conditions:
 \begin{enumerate}
  \item  for all $x\in P_n$, $\extend{f}(\geninc_n x)=f(x)$;
  \item $\extend{f}(x \comp{i} y)=\extend{f}(x)+\extend{f}(y)$, for any $i$\nbd-composable pairs 
    $(x,y)$ of $n$\nbd-cells of $C$ with $i<n$;
  \item $\extend f(\unit n(z))=0$, for any $(n-1)$\nbd-cell $z$.
  \end{enumerate}
\end{lemma}
\begin{proof}
 For the sake of clarity, we suppose that $n>k$. The proof immediately
 adapts to the case $n\leq k$.
Consider the $n$\nbd-groupoid $\B^nG$,
whose set of $i$\nbd-cell is a singleton for $i<n$ and $G$ for $i=n$,
and all composition operations are given by the group
operation. By definition, $\B^nG$ is an \pk{n}{k}-category. The map
$f\colon P_n\to G$ amounts to a morphism of cellular extensions
\[\clx{\trk{n-1}C}{P_n}\to \fgf{n-1,k}(\B^nG).\]
By the universal property of~\cref{prop:univprop}, this induces a
unique map
\[
  \extend f\colon C_n\to G
  \]
satisfying the conditions (1),
(2) and (3).  
\end{proof}
\begin{paragr}
  Let $P$ be an \ook{k}-polygraph, $C=\frgp P$, and $n>0$. Consider
  \[f\colon P_n\to \Z P_n\]
  the canonical injection. By \cref{lemma:delooping} there is a unique
  map $\extend f\colon C_n\to \Z P_n$ satisfying the conditions (1),
  (2) and (3) of the lemma. In particular $\extend f$ is compatible with 
  the relation $\sim_n$ and we get a unique $\Z$-linear map
  \[u\colon \abelk{k}_n(C)\to \Z P_n\]
  such that the following diagram commutes
  \[
    \begin{tikzcd}
      C_n \ar[rd,"\extend{f}"description]\ar[r,"s"]& \abelk{k}_n(C)\ar[d,"u"]\\
      P_n \ar[r,"f"'] \ar[u,"\geninc_n"]& \Z P_n
    \end{tikzcd},
  \]
  with $s:x\mapsto [x]$ the canonical map. On the other hand, by the
  universal property of $\Z P_n$, the map $g=s\circ\geninc_n$ has a unique
  $\Z$-linear extension
  \[v\colon \Z P_n\to\abelk{k}_n(C),\]
  that is, $vf=g$. Note that $s$ is then precisely $\extend g$ as
  defined in~\cref{lemma:delooping}. Now $uvf=us\geninc_n=\extend{f}\geninc_n=f$, which implies that $uv=\Unit{\Z P_n}$. Also
  $vus\geninc_n=v\extend{f}\geninc_n=vf=s\geninc_n$. By the uniqueness
  part of \cref{lemma:delooping}, this implies $vus=s=\extend{g}$,
  whence $vu=\Unit{\abelk{k}_n(C)}$. This proves the following lemma.
 \end{paragr}
\begin{lemma}\label{lemma:abelpoly}
  The morphism of abelian groups $u \colon \abelk{k}_n(C) \to \Z P_n$ is an
  isomorphism, with inverse $v \colon \Z P_n \to \abelk{k}_n(C)$.
\end{lemma}
\begin{paragr}
  Let $k \geq 0$, $P$ be an \ook k-polygraph and $C=\frgp P$ be the free
  \ook k-category on $P$. We denote by $\pair{\Z P}{\partial}$ the chain complex
  \[
    \Z P_0 \overset{\partial_0}{\longleftarrow} \Z P_1
    \overset{\partial_1}{\longleftarrow} \dots,
    \]
  where $\partial_n\colon \Z P_{n+1}\to \Z P_n$ takes a generator $x\in
  P_{n+1}$ to
  \[\partial_n x=u([\target\geninc_{n+1} x])-u([\source\geninc_{n+1}
    x]),\]
  using the isomorphism $u$ of \cref{lemma:abelpoly}. 
\end{paragr}
\begin{proposition}\label{prop:isocomplex}
  The chain complex $\abelk{k}(C)$
  is canonically isomorphic to~$\pair{\Z P}{\partial}$.
  \end{proposition}

\begin{proof}
  By construction $\abelk{k}_0(C)=\Z P_0$. The rest follows
  from \cref{lemma:abelpoly}.
\end{proof}
\begin{paragr}\label{paragr:abeldisk}
  
  As an application of~\cref{prop:isocomplex}, for $k=\oo$ and $n\geq
  0$,  we obtain the following description of the chain complex $\abel(\Dsk{n})$.

  \[
    \abel_i(\Dsk{n})\cong \begin{cases}
      \Z\oplus\Z &\text{ if } 0 \leq i < n,\\
      \Z &\text{ if } i=n,\\
        0 &\text{ if } i>n. \end{cases}
  \]
  For $n>0$ the differential $\partial
  \colon \abel_i(\Dsk{n}) \to \abel_{i-1}(\Dsk{n})$ sends $1$ to
  $(1,-1)$ for $i=n$ and is given by the matrix $\begin{pmatrix} 1 & 1
    \\ -1 & -1 \end{pmatrix}$, 
  for $1 \leq i < n$.

  For $\Sph{n}$, we get $\abel_{n+1}(\Sph{n})=0$ and, for all $i\neq n+1$, $\abel_i(\Sph{n})=\abel_i(\Dsk{n+1})$. 
\end{paragr}

\begin{paragr}\label{paragr:abelloc}
  Let $0 \leq k < \oo$. By definition,
  we have a commutative triangle
  \[
    \begin{tikzcd}[column sep=small]
      \opCat{k}\ar[rr,"\oemb{k}"]\ar[rd,"\abelk{k}"']&& \oCat\ar[dl,"\abel"] \\
      &\Chp.&
  \end{tikzcd}
\]
Now the triangle
 \[
    \begin{tikzcd}[column sep=small]
      \opCat{k}\ar[rd,"\abelk{k}"']&& \oCat\ar[dl,"\abel"] \ar[ll,"\opifun{k}"']\\
      &\Chp&
  \end{tikzcd}
\]
does not {\em a priori} commute. However, it commutes up to a natural
isomorphism, as shown below.
\end{paragr}
\begin{lemma}\label{lemma:abelpi}
  The functors $\abel$ and $\abelk{k}\opifun{k}$ are naturally isomorphic.
\end{lemma}
\begin{proof}
The abelianization functor $\abel$ admits a right adjoint $\mu \colon
\Chp \to \oCat$. For any chain complex $K$, the
\oo-category $\mu(K)$ is in fact an \oo\nbd-groupoid (see \cite[Paragraph
  17.4.3]{abgmmm:polybk}). Therefore, $\mu$ factors uniquely through  $\opCat{k}$ as
  \[
    \Chp \overset{\mu^k}{\longrightarrow} \opCat{k}
    \overset{\oemb{k}}{\longrightarrow} \oCat.
  \] 
  For any $C\in\opCat{k}$ and $K\in\Chp$
  \begin{align*}
    \homset{\Chp}{\abelk{k}(C)}{K} &\cong
                                     \homset{\Chp}{\abel\oemb{k}(C)}{K}\\
                                                     &\cong
                                                       \homset{\oCat}{\oemb{k}(C)}{\mu(K)}\\
                                                     & \cong
                                                       \homset{\oCat}{\oemb{k}(C)}{\oemb{k}\mu^k(K)}\\
                                   &\cong \homset{\opCat{k}}{C}{\mu^k(K)}
                                   \end{align*}
  as $\oemb{k}$ is fully faithful. Hence $\abelk{k}$ is left adjoint
  to $\mu^k$. It follows that $\abelk{k}\opifun{k}$ is left adjoint to
  $\oemb{k}\mu^k=\mu$. Now $\abelk{k}\opifun{k}$ and
  $\abel$ are both left adjoint to the same functor $\mu$, whence
  they are naturally isomorphic.
\end{proof}

\subsection{Monoidal structure}\label{subsec:monoidal}

\begin{paragr}\label{paragr:gray}
  Recall that the category $\oCat$ admits a monoidal category
  structure with the so-called ``Gray tensor product'',
  which we simply denote by $\gray$, and with unit $\Dsk{0}$ (see, for example,
  \cite[Appendice A]{ara_maltsiniotis:joint}). More generally, for any $0 \leq k \leq \oo$, we have
  the induced tensor product $\grayk{k}$ on $\opCat{k}$
  (see~\cite[Section 6]{ara_lucas:folmon}) defined by the
  formula
  \[
    C \grayk{k} D :=\opifun{k}(\oemb{k}C\gray \oemb{k}D).
  \]
  
  Since the unit $\Dsk{0}$ is also the terminal object of $\opCat{k}$,
  for any \ook k\nbd-categories $C$ and $D$, there is a canonical
  morphism
  \[
    C\grayk{k} D \to C \times D,
  \]
  which is natural in $C$ and $D$. This is not an isomorphism unless
  $C$ or $D$ is a $0$\nbd-category (that is, a set).

  A pair $u, v \colon C \to D$ of morphisms in $\opCat{k}$ amounts to
  a morphism
  \[\pair{u}{v}\colon \Sph{0}\grayk{k}C\to D\]
  using the canonical isomorphisms
   \[
    \begin{aligned}
      \Sph{0}\grayk{k}C &=(\Dsk{0}\coprod \Dsk{0})\grayk{k}C \\
      &\cong \left(\Dsk{0}\grayk{k}C\right)\coprod
        \left(\Dsk{0}\grayk{k}C\right) \\
      &\cong C \coprod C.
    \end{aligned}
  \]
  Here the second line uses the fact that the Gray tensor product is
  bi-closed. On the other hand, $\opifun{k}(\Sph 0)=\Sph 0$ and the
  canonical inclusion $\gcof 1\colon\Sph 0\to \Dsk 1$ yields  a
  morphism $\Sph 0\to \opifun{k}(\Dsk 1)$. Thus we may set the
  following definition.
\end{paragr}
\begin{definition}\label{def:oplax}
  Let $u,v\colon C\to D$ be a pair of morphisms in $\opCat k$. An {\em
    oplax transformation} $\alpha\colon u \Rightarrow v$ is a morphism
$ \alpha \colon \opifun{k}(\Dsk{1})\grayk{k}C \to D$
  such that the following diagram commutes:
  \[
    \begin{tikzcd}[column sep=large]
      \Sph{0}\grayk{k} C \ar[d] \ar[dr,"{\left(u,v\right)}"] & \\
      \opifun{k}(\Dsk{1})\grayk{k}C \ar[r,"\alpha"'] & D. 
    \end{tikzcd}
  \]
\end{definition}
  
\begin{paragr}
Let $u,v\colon C\to D$ be morphisms in $\opCat k$, $\alpha\colon
u\Rightarrow v$ be an oplax transformation, and $f\colon D\to D'$ be a
morphism in $\opCat  k$.  The following composition
\[
 \opifun{k}(\Dsk{1}) \grayk{k} C \overset{\alpha}{\longrightarrow} D
  \overset{f}{\longrightarrow} D'.
\]
defines an oplax transformation $f\wsk\alpha\colon fu \Rightarrow fv$.

Similarly, if $g \colon C' \to C$ is a morphism in $\opCat{k}$, the composition
\[
  \begin{tikzcd}
    \opifun{k}(\Dsk{1})\grayk{k} C'
    \ar[rr,"\Unit{\opifun{k}(\Dsk{1})}\grayk{k}g"]&&
    \opifun{k}(\Dsk{1})\grayk{k} C\ar[r,"\alpha"']&
    D
  \end{tikzcd}
\]
defines an oplax transformation $\alpha \wsk g \colon ug \Rightarrow vg$.

Finally, we call \emph{natural transformation from $u$ to $v$} a morphism
\[
  \alpha \colon \opifun{k}(\Dsk{1})\times C \to D
\]
such that the following diagram is commutative
  \[
    \begin{tikzcd}[column sep=large]
      \Sph{0}\times C \ar[d] \ar[dr,"{\left(u,v\right)}"] & \\
      \opifun{k}(\Dsk{1})\times C \ar[r,"\alpha"'] & D. 
    \end{tikzcd}
  \]
  As before, we use here the canonical identification
  $\Sph{0}\times C\cong C \coprod C$. Note that, when $C$ and $D$ are
  $1$-categories, this coincides with the usual notion.

  Any  natural
  transformation $\gamma$ defines an oplax transformation given by the composition
\[
     \opifun{k}(\Dsk{1})\grayk{k}C \longrightarrow  \opifun{k}(\Dsk{1})\times C \overset{\gamma}{\longrightarrow} D.
  \]
\end{paragr}
\begin{proposition}\label{prop:abelmonoidal}
  For any $0 \leq k \leq \oo$, the abelianization functor
  \[
    \abelk{k} \colon \opCat{k} \to \Chp
  \]
  is strong monoidal with respect to the Gray tensor product $\grayk{k}$ on
  $\opCat{k}$ and the usual tensor product of chain complexes on $\Chp$.
\end{proposition}
\begin{proof}
  The case $k=\oo$ is a consequence of~\cite[Proposition A.19]{ara_maltsiniotis:joint}. The
  general case follows from  $\abelk{k}\cong
  \abel\opifun{k}$ (\cref{lemma:abelpi}), together with the fact
  that $\abel$ and $\opifun{k}$ are strong monoidal.
\end{proof}
\begin{corollary}\label{cor:abeltrans}
  For any $0 \leq k \leq \oo$, the abelianization functor
    \[
    \abelk{k} \colon \opCat{k} \to \Chp
  \]
  sends oplax transformations to chain homotopies. 
\end{corollary}
\begin{proof}
  Let $u,v\colon C\to D$ be morphisms in $\opCat k$ and $\alpha:
  u\Rightarrow v$ be an oplax transformation. By applying the functor
  $\abelk k$ to the diagram
   \[
    \begin{tikzcd}[column sep=large]
      \Sph{0}\grayk{k} C \ar[d] \ar[dr,"{\left(u,v\right)}"] & \\
      \opifun{k}(\Dsk{1})\grayk{k}C \ar[r,"\alpha"'] & D
    \end{tikzcd}
  \]
  and, using \cref{prop:abelmonoidal}, we get a commutative diagram in
  $\Chp$
  \[
  \begin{tikzcd}[column sep=large]
      \abelk{k}(\Sph{0})\otimes\abelk{k}(C) \ar[d] \ar[dr,"\abelk{k}\pair{u}{v}"] & \\
      \abelk{k}\opifun{k}(\Dsk{1})\otimes\abelk{k}(C) \ar[r,"\abelk{k}(\alpha)"'] & \abelk{k}(D),
    \end{tikzcd}
  \] 
which, by \cref{lemma:abelpi} and \ref{paragr:abeldisk}, becomes
\[
    \begin{tikzcd}[column sep=large]
     \abelk{k}(C)\oplus \abelk{k}(C) \ar[d] \ar[dr,"{\left(\abelk{k}(u),\abelk{k}(v)\right)}"] & \\
      \abel(\Dsk{1})\otimes\abelk{k}(C) \ar[r,"\abelk{k}(\alpha)"'] & \abelk{k}(D).
    \end{tikzcd}
  \]
 By the description of $\abel(\Dsk{1})$ from \ref{paragr:abeldisk},
  this amounts exactly to a chain homotopy from $\abelk{k}(u)$ to $\abelk{k}(v)$.
\end{proof}

\subsection{Deriving the abelianization functor}\label{subsec:abelderiv}
\begin{proposition}\label{prop:abelquillen}
  For any $k \geq 0$, the abelianization functor
  \[
    \abelk{k} \colon \opCat{k} \to \Chp
  \]
  is left Quillen with respect to the folk model structure on
  $\opCat k$ and the projective model structure on $\Chp$.
\end{proposition}
\begin{proof}
   Let $I$ (resp.\ $J$) be a set of generating cofibrations (resp.\ generating
  trivial cofibrations) in $\oCat$. Since the folk model structure on
  $\opCat k$ is right induced from
  the inclusion functor $\oemb{k} \colon \opCat{k} \to \oCat$,
  $\opifun k(I)$ and $\opifun k(J)$ are respectively, sets of
  generating cofibrations and generating trivial cofibrations in
  $\opCat k$. By \cref{lemma:abelpi}, $\abelk k\opifun k(I)\cong
  \abel(I)$ and  $\abelk k\opifun k(J)\cong \abel(J)$. Thus it
  suffices to prove that $\abel (I)$ and $\abel (J)$ are respectively
  sets of cofibrations and weak equivalences in $\Chp$ and this follows
  from~\cite[Theorem 22.2.7]{abgmmm:polybk}. 
\end{proof}

\begin{paragr}\label{paragr:subtlety}
  By definition of $\abelk k$, there is a commutative triangle
  \[
    \begin{tikzcd}[column sep=small]
      \opCat{k}\ar[rr,"\oemb{k}"]\ar[rd,"\abelk{k}"']&& \oCat\ar[dl,"\abel"] \\
      &\Chp.&
  \end{tikzcd}
\]
Since the embedding functor $\oemb{k}$ preserves the weak
equivalences, it induces a functor
\[\oemb{k}\colon \Loc{\opCat{k}}\to  \Loc{\oCat}\]
at the level of homotopy categories. However, this does {\em not}
imply that the triangle
\[
  \begin{tikzcd}[column sep=small]
    \Loc{\opCat{k}}\ar[rr,"\oemb{k}"]\ar[rd,"\LL \abelk{k}"']&&
    \Loc{\oCat}\ar[dl,"\LL \abel"] \\
    &\Loc{\Chp}&
  \end{tikzcd}
\]
commutes. The universal property of left derived functors only yields a natural transformation
\begin{equation}\label{transnat}
  \homnat\colon\LL \abel\circ \oemb{k} \Rightarrow \LL \abelk{k}.
\end{equation}
Let $C$ be an \ook k-category and
\[
  p\colon\frgp P\to C
\]
a resolution of $C$ by an \ook k-polygraph $P$. As $\frgp P$ is
cofibrant in $\opCat k$,
\[
  \Loc{\abelk k}(C)\cong \abelk k(\frgp P).
  \]
On
the other hand, $\oemb{k}(\frgp P)$ is generally {\em not} cofibrant
in $\oCat$, but admits a cofibrant replacement
\[q\colon \frcat Q\to \oemb{k}(\frgp P).\]
Thus $\LL\abel(\oemb{k}(C))\cong \abel(\frcat Q)$ and the component
$\homnat_C$ of~(\ref{transnat}) is
\[\abel(q):\abel(\frcat Q)\to \abelk{k}(\frgp P)\]
up to isomorphisms depending on the choices of resolutions. Note that
resolutions in $\oCat$ as well as in $\opCat k$ may be chosen
functorially in $C$ and therefore the above isomorphisms can be chosen
natural in $C$. 
\end{paragr}
\begin{definition}\label{def:hompol}
  Let $k\geq 0$. The \emph{$\ook k$\nbd-polygraphic homology} of
  an \ook k\nbd-category $C$ is defined as
  \[
    \Hop{k}(C):=\LL \abelk{k} (C).
  \]
\end{definition}

\begin{paragr}\label{paragr:mainquestion}
  Let $C$ be an \ook k-category. Besides its \ook k-polygraphic
  homology, we may also consider its polygraphic homology when $C$ is regarded
  as an \oo-category. More precisely, we shall denote by $\Hopol(C)$ the
  \ook{\oo}-polygraphic homology of $\oemb{k}(C)$, that is,
  $\Hop{\oo}(\oemb{k}(C))$. The main question we address in this paper
  is then whether the natural morphism
  \[\homnat_C\colon \Hopol(C)\to \Hop{k}(C)  \]
  is an isomorphism.

  Whereas the general question remains open,
  we shall prove that $\homnat_C$ is indeed an isomorphism in two
  useful cases:
  \begin{itemize}
  \item $C$ is a $1$-category, seen as an \ook k-category for any $k\geq
    1$;
  \item $C$ is a groupoid, seen as an \ook k-category for any $k\geq
    0$.
  \end{itemize}
  In the course of the proof, we first recall a few fundamental
  properties of the usual homology of small categories, and finally
  compare the latter with its polygraphic counterpart.
\end{paragr}


%% file: htpykan.tex
\section{A recollection of homotopy Kan extensions in model categories}\label{sec:htpykan}
\subsection{A simplifying technical hypothesis}
\begin{paragr}\label{paragr:techyp}
  All the Quillen adjunctions $F \colon \M \leftrightarrows \M' \colon G$
  that show up in this paper have the extra property that $G$
  preserves \emph{all} weak equivalences. This is, for instance, the
  case when all objects of $\M'$ are fibrant.
This simplifies certain coherence issues: $G$ induces a \emph{unique} functor $\GZ{G}$ at the level of
  homotopy categories such that the following square is \emph{genuinely}
  commutative
  \[
    \begin{tikzcd} \M' \ar[d] \ar[r,"G"] &
\M\ar[d]\\ \Loc{\M'} \ar[r,"\GZ{G}"] &
\Loc{\M},
    \end{tikzcd}
  \]
  where $\Loc{\M}$ and $\Loc{\M'}$ are the localizations with respect
  to the weak equivalences in the sense of Gabriel--Zisman \cite{gabrielzisman:calfrh}, and
  the vertical arrows are the localization functors. It follows from
  the universal property of localization that
  $\GZ{G}$ is a right derived
  functor of $G$ and we can define $\mathbb{R}G=\overline{G}$. In addition to giving a canonical choice of right derived functor, this construction
  has a useful strict $2$\nbd-functorial property. Given another
  Quillen adjunction $F' \colon \M \leftrightarrows \M' \colon G'$, 
  any natural transformation \[\alpha \colon G \Rightarrow G'\] induces
  a unique natural transformation \[\GZ{\alpha} \colon \GZ{G}
    \Rightarrow \GZ{G'}\]
  that makes the appropriate $2$\nbd-diagram commute. It follows that the
  construction
  \[
    \begin{aligned}
    \M &\longmapsto\Loc{\M}\\
    G &\longmapsto\GZ{G}\\
    \alpha &\longmapsto\GZ{\alpha}\\
  \end{aligned}
\]
defines a \emph{strict} $2$\nbd-functor from the $2$\nbd-category of
  model categories, right Quillen functors that preserve all weak
  equivalences and natural transformations between them, to the
  $2$\nbd-category of (large) categories.

  In contrast, there is no canonical choice for the total left derived functor $\LL F$ associated with
  a left Quillen functor $F$. Furthermore, even if we make a choice for any such
  functor, and consider the $2$\nbd-functoriality which associates with
  any natural transformation $\beta \colon F \Rightarrow F'$ the (unique) natural
  transformation $\LL \beta \colon \LL F \Rightarrow \LL F'$ making
  the relevant $2$\nbd-diagram commute, the
  following correspondence
  \[
  \begin{aligned}
    \M &\longmapsto\Loc{\M}\\
    F &\longmapsto\LL F\\
    \beta &\longmapsto \LL \beta\\
  \end{aligned}
 \] 
  is only \emph{pseudo}-functorial. Hence, it involves coherence 
  data which we have to keep track of, making some proofs cumbersome.

  The upshot is that since each $\LL F$ is left adjoint to $\GZ{G}$,
  we can often exploit this fact to reduce a proof concerning left
  adjoints to one concerning right adjoints.
  This turns out to be much simpler thanks to the \emph{strict}
  $2$\nbd-functoriality stated above.
    \end{paragr}
\subsection{Projective model structure}\label{projms}

\begin{paragr}\label{paragr:htpykan} Let $\M$ be a cofibrantly generated model
category. Recall that for any small category $I$, we can equip the
category of functors $\fun{I}{\M}$ with the \emph{projective model
structure}, where the weak equivalences and the fibrations are the
natural transformations that are pointwise weak equivalences and
fibrations in $\M$, respectively. Any functor $u \colon I \to J$
between small categories induces, by precomposition, a functor
  \begin{align*} u^* \colon \fun{J}{\M} &\to \fun{I}{\M}\\ X &\mapsto
X\circ u.
  \end{align*} This functor admits a left adjoint $u_!
\colon\fun{I}{\M} \to \fun{J}{\M} $, which is given pointwise by the Kan
extension formula:
  \[ u_!(X)(j)\cong\ilim_{\tr{I}{j}}X\vert_{\tr{I}{j}}.
  \] By construction, $u^*$ preserves the weak equivalences and the
fibrations of the projective model structure. Hence, the adjunction
$u_! \dashv u^*$ is a Quillen adjunction and we get an adjunction at
the level of homotopy categories
  \[ \LL u_! \colon \Loc{\fun{I}{\M}} \leftrightarrows
\Loc{\fun{J}{\M}} \colon \GZ{u^*}.
  \]
\end{paragr}
\begin{paragr}\label{paragr:hocolim} We denote by $\Term$ the terminal category and make the
canonical identification $\fun{\Term}{\M}=\M$. For a small category
$I$, we denote by $p_I \colon I \to \Term$ the unique such
functor. The functor $p_I^* \colon \M \to \fun{I}{\M}$ is the
``diagonal'' functor that sends an object $X$ of $\M$ to the
constant diagram $I \to \M$ with value $X$. Hence, $p_{I!}$ is nothing
but the colimit functor $\ilim_I \colon \fun{I}{\M} \to \M$, and we
shall use the two notations interchangeably. Consequently, $\LL
p_{I!}$ is nothing but the \emph{homotopy colimit functor}
\[
  \hocolim_I \colon \Loc{\fun{I}{\M}} \to \Loc{\M}.
  \]
\end{paragr}
\begin{paragr} Let $I$ be a small category and $i \in \Ob(I)$ be an
object of $I$. We can regard this as a functor $i \colon \Term \to I$,
which induces an ``evaluation at $i$'' functor
  \begin{align*} i^* \colon \fun{I}{\M} &\to \M\\ X &\mapsto X(i).
  \end{align*} It is a standard fact of category theory that the
family of functors
  \[\left(i^*\colon \fun{I}{\M} \to \M\right)_{i\in \Ob(I)}\] is
\emph{conservative}, meaning that a morphism $\alpha \colon X \to Y$
of $\fun{I}{\M}$ is an isomorphism if and only if for each object $i$
of $I$, $i^*(\alpha)=\alpha_i$ is an isomorphism. The same is true at
the level of homotopy categories as we recall in the following result.
\end{paragr}
\begin{proposition}\label{prop:der2} Let $\M$ be a (cofibrantly generated) model
category and $I$ be a small category. The family of functors
  \[ \left(\GZ{i^*} \colon \Loc{\fun{I}{\M}}\to \Loc{\M}\right)_{i \in
\Ob(I)}
  \] is conservative.
\end{proposition}
\begin{proof}
  See \cite[Lemme 2.4]{cisinski:images}.
\end{proof}
\subsection{Homotopy cocontinuous functors}\label{hptycocts}
\begin{paragr}\label{paragr:univmor} Let $F \colon \M \leftrightarrows \M' \colon G$ be a
  Quillen adjunction between (cofibrantly generated) model categories
  and such that $G$ preserves all weak equivalences. For any small
category $I$, we consider the functors induced by postcomposition
  \[\begin{aligned} \fun{I}{F} \colon \fun{I}{\M} &\to \fun{I}{\M'}\\ X &\mapsto
F\circ X.
  \end{aligned} \quad \quad \begin{aligned} \fun{I}{G} \colon \fun{I}{\M'} &\to \fun{I}{\M}\\ X &\mapsto
G\circ X.
  \end{aligned}\]
The adjunction $\fun{I}{F} \colon \fun{I}{\M} \leftrightarrows
\fun{I}{\M'} \colon \fun{I}{G}$ is again a Quillen adjunction with
respect to the projective model structures. Furthermore, $\fun{I}{G}$
preserves all weak equivalences. Hence, it induces an adjunction at
the level of homotopy categories

  \[ \LL \fun{I}{F} \colon \Loc{\fun{I}{\M}}\leftrightarrows
    \Loc{\fun{I}{\M'}} \colon \GZ{\fun{I}{G}}.
  \]
Now let $u \colon I \to J$ be a functor between small
categories. By the strict $2$\nbd-functoriality property stated in
\ref{paragr:techyp}, we have a strictly commutative diagram
  \[
    \begin{tikzcd} \Loc{\fun{J}{\M'}} \ar[r,"\GZ{\fun{J}{G}}"] \ar[d,"\GZ{u^*}"']&
\Loc{\fun{J}{\M}} \ar[d,"\GZ{u^*}"]\\ \Loc{\fun{I}{\M'}} \ar[r,"\GZ{\fun{I}{G}}"'] & \Loc{\fun{I}{\M}}.
    \end{tikzcd}
  \]
  By adjunction, we can then define a canonical natural transformation
  \[
     \LL \fun{I}{F}  \GZ{u^*} \Rightarrow \GZ{u^*}  \LL \fun{J}{F}
   \]
   as the following composition
     \[
       \begin{tikzcd}
         \Loc{\fun{J}{\M}}\ar[d,"\LL \fun{J}{F}"']
         \ar[rd,"\Unit{\Loc{\fun{J}{\M}}}",""{name=A,below left}] & \\
         \Loc{\fun{J}{\M'}} \ar[r,"\GZ{\fun{J}{G}}"'] \ar[d,"\GZ{u^*}"']&
\Loc{\fun{J}{\M}} \ar[d,"\GZ{u^*}"]\\ \Loc{\fun{I}{\M'}}
\ar[r,"\GZ{\fun{I}{G}}"]\ar[rd,"\Unit{\Loc{\fun{I}{\M'}}}"',""{name=B,above
right}] & \Loc{\fun{I}{\M}}\ar[d,"\LL \fun{I}{F}"]\\
&\Loc{\fun{I}{\M'}}.
\ar[from=A,to=2-1,Rightarrow,"\eta"]
\ar[from=3-2,to=B,Rightarrow,"\epsilon"]
    \end{tikzcd}
  \]
  This, in turn, allows us to define a canonical natural transformation
  \[
    \LL u_!\LL \fun{I}{F} \Rightarrow \LL \fun{J}{F} \LL u_! 
  \]
  as the following composition
  \[
    \begin{tikzcd}
    &  \Loc{\fun{J}{\M'}}\ar[d,"\GZ{u^*}"]
    \ar[dl,"\Unit{\Loc{\fun{J}{\M'}}}"',""{name=A,below right}]& \Loc{\fun{J}{\M}} \ar[l,"\LL
    \fun{J}{F}"'] \ar[d,"\GZ{u^*}"'] & \Loc{\fun{I}{\M}} \ar[l,"\LL u_!"']
    \ar[dl,"\Unit{\Loc{\fun{I}{\M}}}",""{name=B,above left}]
      \\
    \Loc{\fun{J}{\M'}} & \ar[l,"\LL u_!"]
    \Loc{\fun{I}{\M'}}&\Loc{\fun{I}{\M}}. \ar[l,"\LL \fun{I}{F}"]
    \ar[from=2-2,to=A,Rightarrow,"\epsilon"]
    \ar[from=B,to=1-3,Rightarrow,"\eta"]
    \ar[from=2-3,to=1-2,Rightarrow]
    \end{tikzcd}
  \]  
\end{paragr}
\begin{proposition}\label{prop:htpycocts} Let $\M$ and $\M'$ be two (cofibrantly generated)
  model categories, $F \colon \M \leftrightarrows \M' \colon G$ be a Quillen adjunction such
  that $G$ preserves all weak equivalences, and
  $u \colon I \to J$ be a functor between small categories. Then the
  canonical natural transformations
  \begin{equation}\label{eq:morphder}
     \LL \fun{I}{F}  \GZ{u^*} \Rightarrow \GZ{u^*}  \LL \fun{J}{F}
   \end{equation}
   and
   \begin{equation}\label{eq:cocontinuous}
     \LL u_!\LL \fun{I}{F} \Rightarrow \LL \fun{J}{F} \LL u_! 
   \end{equation}
   are isomorphisms, which are natural in $F$ in the obvious sense.
\end{proposition}
\begin{proof}
 The fact that they are isomorphisms is a particular case of the dual
 of \cite[Proposition 6.12]{cisinski:images}. For the naturality, a
 calculus of mates argument shows that this is implied by the naturality
 in $G$
 of the equality
 \[
   \GZ{\fun{I}{G}}\GZ{u^*}=\GZ{u^*}\GZ{\fun{J}{G}},
 \]
 which is a straightforward consequence of the strict
 $2$\nbd-functoriality property stated in \ref{paragr:techyp}. Details
 are left to the reader.
\end{proof}
\begin{remark}
  The fact that \eqref{eq:cocontinuous} is an isomorphism is commonly
  referred to as the fact that \emph{left Quillen functors preserve
    homotopy left Kan extensions}. In particular, taking $J=\Term$,
  $u=p_I \colon I \to \Term$ the unique such functor and $X$ an
  object of $\Loc{\fun{I}{\M}}$, we get an isomorphism of $\Loc{\M}$
  \begin{equation}\label{eq:htpycolim}
    \hocolim_I \LL \fun{I}{F}(X) \cong \LL F (\hocolim_I X),
  \end{equation}
  natural in $X$ and in $F$.
\end{remark}
\subsection{Abstract homology of small categories}\label{abstracthom}
\begin{definition}\label{def:abstracthom}
  Let $\M$ be a (cofibrantly generated) model category. For a small
  category $I$ and $X$ an object of $\M$, we define the \emph{$\M$\nbd-valued
    homology of $I$ with constant coefficients $X$} as the object of $\Loc{\M}$
  \[
    \Ho_{\M}(I,X):=\LL p_{I!}\GZ{p_I^*}X=\hocolim_I (\GZ{p_I^*}X).
  \]
\end{definition}

\begin{remark}
  Note that the use of the term ``homology'' here is arguably slightly
  abusive as we don't make any hypothesis of additivity or stability
  property of any kind on (the localization of) $\M$. Later, we will
  focus on the case where $\M$ is the category of non-negatively
  graded chain complexes, whence recovering the usual setting of homology.
\end{remark}
\begin{paragr}
 Let us explore the functoriality of $\Ho_{\M}(I,X)$ in $I$. Let
  $u \colon I \to J$ be a functor between small categories. The
  commutative triangle
  \[
    \begin{tikzcd}[column sep=small]
      I \ar[rr,"u"]\ar[rd,"p_I"'] && J\ar[ld,"p_J"] \\
      &\Term&
    \end{tikzcd}
  \]
  induces a (strictly) commutative triangle
  \[
    \begin{tikzcd}[column sep=small]
      \Loc{\fun{I}{\M}}  && \Loc{\fun{J}{\M}}\ar[ll,"\GZ{u^*}"'] \\
      &\Loc{\M}.\ar[lu,"\GZ{p_I^*}"]\ar[ru,"\GZ{p_J^*}"']&
    \end{tikzcd}
  \]
  Hence, by uniqueness of adjoints, this yields a triangle that commutes
  up to a canonical isomorphism
  \[
        \begin{tikzcd}[column sep=small]
      \Loc{\fun{I}{\M}} \ar[rr,"\LL u_!"]\ar[rd,"\LL p_{I!}"'] &&
      \Loc{\fun{J}{\M}}\ar[ld,"\LL p_{J!}"] \\
      &\Loc{\M}.&
    \end{tikzcd}
  \]
  Now let $X$ be an object of $\M$. There is a canonical morphism
  \[
    \LL p_{I!}\GZ{p_I^*}X =\LL p_{I!}\GZ{u^*p_J^*}X \cong \LL p_{J!} \LL u_! \GZ{u^*p_J^*}X\to \LL p_{J!}\GZ{p_J^*}X,
  \]
  where the morphism on the right is the co-unit of the adjunction
  $\LL u_! \dashv \GZ{u^*}$. In other words, this defines a morphism
  \[
  \Ho_{\M}(u,X)\colon  \Ho_{\M}(I,X)\to \Ho_{\M}(J,X).
  \]
  We then have the following functoriality result, whose proof is a
  straightforward verification left to the reader.
\end{paragr}
\begin{proposition}\label{prop:functorialityhomology}
  Let $\M$ be a (cofibrantly generated) model category and $X$ be an
  object of $\M$. We have a functor
  \[
    \Ho_{\M}(-,X) \colon \Cat \to \Loc{\M},
  \]
  whose action on objects is
  \[
    I \longmapsto \Ho_{\M}(I,X),
  \]
  and whose action on morphisms is
  \[
    u \colon I \to J \longmapsto \Ho_{\M}(u,X) \colon
    \Ho_{\M}(I,X) \to \Ho_{\M}(J,X).
  \]
\end{proposition}
\subsection{Classical homology of categories}
\begin{paragr}\label{paragr:homalg}
  For $C$ a small category, we denote by $\LMod{C}$ the category of
  \emph{left $C$\nbd-modules}, that is of functors
  \[
    C \to \Ab
  \]
  from $C$ to the category of abelian groups, and of natural
  transformations between them. By definition, the category of \emph{right}
  $C$\nbd-modules is the category
  \[
    \RMod{C}:=\LMod{C^{\op}}
  \]
  and everything that follows is dualized accordingly. For an object
  $X$ of $\Ab$, we shall abusively denote by $\cst{X}$ both the constant left
  and right $C$\nbd-modules with value $X$. Using the notations from
  \ref{paragr:hocolim}, this means that $\cst{X}$ is a common notation for
  \[
   p_C^*X \text{ and } p_{C^\op}^*X.
  \]
  We denote by $\Chp(C)$ the category of chain complexes
  of left $C$\nbd-modules in non-negative degree. For $C=\Term$ the
  terminal category, we use the notation $\Chp(\Z)$ instead. 
  We denote by $\Derp{C}$ the \emph{derived
 category of $C$} in non-negative degree, that is the
localization of $\Chp(C)$ with respect to
 quasi-isomorphisms (and we use the notation $\Derp{\Z}$ when
 $C=\Term$ is the terminal category). This localization is controlled by the
  so-called \emph{projective} model structure on $\Chp(C)$, which is
  cofibrantly generated,
  characterized as follows:
  \begin{itemize}
  \item the weak equivalences are the quasi-isomorphisms,
  \item the fibrations are the levelwise epimorphisms in each positive
    degree,
  \item the cofibrations are the levelwise monomorphisms with
    projective cokernel in each non-negative degree.
  \end{itemize}
  In particular, the cofibrant objects are the chain complexes (in
  non-negative degree) of degreewise projective
  $C$\nbd-modules. The following lemma is useful to
  construct and recognize such objects.
\end{paragr}
\begin{lemma}\label{lemma:freeisproj}
  Let $c$ be an object of $C$. The left $C$\nbd-module $\Z[\homset{C}{c}{-}]$
  \[
    \begin{aligned}
      C &\to \Ab\\
      d &\mapsto \Z[\homset{C}{c}{d}],
    \end{aligned}
  \]
  is a projective object of $\LMod{C}$.
\end{lemma}
\begin{proof}
  This follows straightforwardly from the Yoneda lemma.
\end{proof}

\begin{paragr}
   Notice that for a small category $C$, there is a canonical
   \emph{isomorphism} of categories
  \[
    \Chp(C)\cong \fun{C}{\Chp(\Z)},
  \]
  and we shall use this identification implicitly. Note that this
  identification is compatible with model structures in the sense that
  the projective model structure on $\fun{C}{\Chp(\Z)}$ in the sense
  of \ref{paragr:htpykan} is the same, via the previous isomorphism,
  as the projective model structure on $\Chp(C)$ in the sense of
  \ref{paragr:homalg}.

  This allows us to use all the results from the sections
  \ref{projms}, \ref{hptycocts} and \ref{abstracthom}. 
\end{paragr}
\begin{definition}\label{def:classicalhomology}
  Let $C$ be a small category. We define
  the \emph{homology of $C$}, and denote it by
    $\Ho(C),$
  as the $\Chp(\Z)$\nbd-valued homology of $C$ with constant coefficients $\cst{\Z}$
  (\cref{def:abstracthom}).
\end{definition}
\begin{remark}\label{remark:homologygroups}
  By definition, $\Ho(C)$ is an object of $\Derp{\Z}$. To recover
  homology groups in the usual sense, it suffices to postcompose with the $n$-th
  homology group functor $H_n \colon \Derp{\Z}\to \Ab$. In other
  words, we can make the following definition
  \[
    \Ho_n(C):=H_n(\Ho(C)).
  \]
\end{remark}
\begin{paragr}
  Let us quickly unfold \cref{def:classicalhomology}. By definition,
  the homology of a small category $C$ is
  \[
    \Ho(C)=\hocolim_C\cst{\Z}.
  \]
  In order to make this formula look more familiar, let us introduce
  the following construction. Given a right $C$\nbd-module $M$ and a left $C$\nbd-module $N$, we define their (total) tensor product as
  the following coend, taken in the category of abelian groups,
  \[
    M\tens{C}N:=\int^{c \in C}M(c)\tens{\Z} N(c).
  \]
  Note that when $C$ is a group, say $C=G$, then this is indeed the usual
  tensor product of $\Z G$\nbd-modules, where $\Z G$ is the group ring
  of $G$. The above tensor product defines a functor
  \[
    -\tens{C}-\colon \RMod{C} \times \LMod{C} \to \Ab.
  \]
  By coend calculus, taking $M=\cst{\Z}$ yields a
  natural isomorphism
  \[
    \cst{\Z}\tens{C}N \cong \ilim_C N.
  \]
  In particular, this implies that we have
  \[
    \Ho(C)\cong \cst{\Z} \Ltens{C} \cst{\Z},
  \]
  where $\Ltens{C}$ is the left derived functor of
  $\tens{C}$. If we define
  \[
    \Tor{C}{n}(M,N):=H_n(M \Ltens{C} N),
  \]
  then, following \cref{remark:homologygroups}, we have
  \[
    H_n(C)\cong \Tor{C}{n}(\cst{\Z},\cst{\Z}).
  \]
  Note that in the case that $C$ is a group $C=G$, we indeed recover
  the usual notion of $\Tor{}{}$, and thus the classical definition of
  group homology.
\end{paragr}


%% file: comparison.tex
\section{Comparison theorem}\label{sec:comparison}

\subsection{The two fundamental properties}
\begin{paragr}\label{paragr:colimtr}
  Let $C$ be a $1$\nbd-category. Recall from \ref{paragr:pullback_resol} that, given an $\ook
  k$\nbd-category
  $X$ and an $\ook k$\nbd-functor $p \colon X \to C$, with $1 \leq k \leq \oo$, we defined a functor
    \[
    \begin{aligned}
      \tr{X}{-} \colon C &\to \opCat{k} \\
      c &\mapsto \tr{X}{c}.
    \end{aligned}
  \]
    This
    construction is functorial in $p\colon X \to C$ as it canonically defines a functor
  \begin{align*}
    \tr{\opCat{k}}{C} &\to \funopCat{k}{C}\\
    (X, X \overset{p}{\to} C) &\mapsto \left(c \mapsto \tr{X}{c}\right).
  \end{align*}
  Notice that for $(X, X \overset{p}{\to} C)$ in $\tr{\opCat{k}}{C}$,
  we have a cocone
  \[
    (\tr{X}{c} \to X)_{c \in \Ob{C}},
  \]
  hence
  a canonical morphism
  \[
    \ilim_{c \in C}\tr{X}{c} \to X,
  \]
  which is natural in the sense that it defines a natural
  transformation from the functor
  \begin{align*}
    \tr{\opCat{k}}{C} &\to \opCat{k}\\
    (X, X \overset{p}{\to} C) &\mapsto \ilim_{c \in C} \tr{X}{c},
  \end{align*}
  to the projection functor
    \begin{align*}
    \tr{\opCat{k}}{C} &\to \opCat{k}\\
    (X, X \overset{p}{\to} C) &\mapsto X.
    \end{align*}
    Finally, if we take $C$ a groupoid and $k=0$, then everything said
    above still makes sense.
\end{paragr}
\begin{lemma}\label{lemma:colim}
  The natural morphism of $\opCat{k}$
  \[
     \ilim_{c \in C}\tr{X}{c} \to X,
   \]
   is an isomorphism.
\end{lemma}
\begin{proof}
The case $k=\oo$ is \cite[Lemma 7.5]{guetta:homcat}. The general case
follows from the fact that $\opCat{k} \to \oCat$
  preserves colimits.
\end{proof}
\begin{proposition}\label{prop:hocolimtr}
  Let $C$ be a $1$\nbd-category. Then for any $1 \leq k \leq \oo$, the canonical morphism of
  $\Loc{\opCat{k}}$
  \[
    \hocolim_{c\in C }\tr{C}{c} \to C
    \]
  is an isomorphism. If $C$ is a groupoid, then this also works for
  $k=0$.
\end{proposition}
\begin{proof}
  Let $\frgp{P} \to C$ be a polygraphic resolution of $C$ in
  $\opCat{k}$ and \[\tr{\frgp P\!}{-} \colon C \to \opCat{k}\] obtained
  as in section \ref{subsec:polres}. We get a commutative diagram in
  $\Loc{\opCat{k}}$
  \[
    \begin{tikzcd}
      \displaystyle\hocolim_{c \in C}\tr{\frgp P\!}{c} \ar[d] \ar[r]& \displaystyle\ilim_{c \in
        C}\tr{\frgp P\!}{c} \ar[d] \ar[r] & \frgp P \ar[d] \\
      \displaystyle\hocolim_{c \in C}\tr{C}{c} \ar[r]&\displaystyle\ilim_{c \in C}\tr{C}{c}\ar[r] & C.
    \end{tikzcd}
  \]
  Note that both $\ilim_{c \in C}$ in the middle are the images by the
  localisation functor \[\opCat{k} \to \Loc{\opCat{k}}\]
  of the actual
  colimits taken in the category $\opCat{k}$. The right vertical
  arrow is an isomorphism by hypothesis, the horizontal arrows of the
  square on the right are isomorphisms by \cref{lemma:colim}. By
  stability of trivial fibrations under pullback, the natural transformation
  $\tr{\frgp P\!}{-} \to \tr{C}{-}$ is a pointwise weak equivalence. Since homotopy colimits
  send pointwise weak equivalences to isomorphisms (in the homotopy
  category), this proves that the left vertical arrow is an
  isomorphism. Finally, the top left
  horizontal arrow is an isomorphism
  because $\tr{\frgp P\!}{-}$ is cofibrant by \cref{thm:projresol} (or
  \cref{thm:projresol_grp} in the case $k=0$). By applying
  two-out-of-three twice, we get that all the remaining arrows of this diagram
  are isomorphisms too. 
\end{proof}
\begin{lemma}\label{lemma:liftoplax}
  Let $0 \leq k \leq \omega$ and consider a diagram as below in $\opCat{k}$, 
  \[
    \begin{tikzcd}[row sep=large]
      A' \ar[d,"p"]\ar[r,bend left,"u'"] \ar[r,bend right,"v'"']& B' \ar[d,"q"] \\
      A \ar[r,bend left,"u",""{name=A,below}] \ar[r,bend
      right,"v"',""{name=B,above}]& B,
      \ar[from=A,to=B,Rightarrow,"\alpha"]
    \end{tikzcd}
  \]
  where the front and back faces are commutative. If $q$ is a trivial
  fibration and $A'$ is cofibrant in the folk model structure on
  $\opCat{k}$, then there exists an oplax transformation $\alpha'
  \colon u' \Rightarrow v'$ making the whole diagram commute, which means
  that
  \[
    q\wsk \alpha' = \alpha \wsk p.
  \]
\end{lemma}
\begin{proof}
  The data of the given diagram amounts to the following commutative
  diagram in $\opCat{k}$
  \[
    \begin{tikzcd}
      \Sph{0} \grayk{k} A' \ar[r]\ar[d] \ar[rr,bend
      left,"{(u',v')}"]& \opifun{k}(\Dsk{1}) \grayk{k} A'
      \ar[d]& B' \ar[d,"q"] \\
      \Sph{0} \grayk{k} A \ar[r] & \opifun{k}(\Dsk{1}) \grayk{k} A \ar[r,"\alpha"]
      & B.
    \end{tikzcd}
  \]
  Since $A'$ is cofibrant and $\grayk{k}$ makes the folk model
  structure on $\opCat{k}$ a monoidal model structure \cite{ara_lucas:folmon}, the
  arrow $ \Sph{0}  \grayk{k} A' \to  \opifun{k}(\Dsk{1}) \grayk{k} A'$ is a
  cofibration. Since $q$ is a trivial fibration, there exists a
  morphism $\alpha' \colon \opifun{k}(\Dsk{1})  \grayk{k} A' \to B'$ making the
  whole diagram commute.
\end{proof}
\begin{proposition}\label{prop:hlgycontrcat}
  Let $C$ be a small category with a terminal object. Then, for any $1
  \leq k \leq \oo$ the
  canonical morphism of $\Derp{\Z}$
  \[
    \Hop{k}(C) \to \Z,
  \]
  induced by the canonical functor $C \to \Term$, is an
  isomorphism. If $C$ is a groupoid, then it is also an isomorphism
  for $k=0$.
\end{proposition}
\begin{proof}
  Let us denote by $c_0$ the terminal object of $C$. We then have a
  natural transformation
    \[
    \begin{tikzcd}
      C \ar[r,bend left,"\Unit{C}",""{name=A,below}]
      \ar[r,bend right,"\cst{c_0}"',""{name=B,above}] & C
      \ar[from=A,to=B,Rightarrow]
    \end{tikzcd}
  \]
  from the identity functor on $C$ to the constant functor with value
  $c_0$.  By \ref{paragr:gray}, this defines an oplax
  transformation. Now, let
  \[
    \frgp{P} \to C
  \]
  be a polygraphic resolution of $C$. Since trivial fibrations are
  surjective on objects, we may choose an object $x_0$ of $P$ whose image in
  $C$ is $c_0$. Thus, we have a diagram
  \[
    \begin{tikzcd}[row sep=huge]
      \frgp{P} \ar[d]\ar[r,bend left,"\Unit{\frgp{P}}"] \ar[r,bend
      right,"\cst{x_0}"']& \frgp{P} \ar[d] \\
      C \ar[r,bend left,"\Unit{C}",""{name=A,below}] \ar[r,bend
      right,"\cst{c_0}"',""{name=B,above}]& C,
      \ar[from=A,to=B,Rightarrow]
    \end{tikzcd}
  \]
  such that the front and back faces commute. Since
  $\frgp{P}$ is cofibrant, and since the vertical arrows are trivial
  fibrations, we can apply \cref{lemma:liftoplax}, which yields an
  oplax transformation from the identity \oo\nbd-functor on $P$ to
  the constant \oo\nbd-functor with value $x_0$.

  Finally, applying \cref{cor:abeltrans}, we get a chain homotopy
  between the identity morphism \[\abelk{k}(P) \to
    \abelk{k}(P)\] and the composition
  \[\abelk{k}(P) \to \Z \to
  \abelk{k}(P),\] which proves that  $\abelk{k}(P) \to
  \Z$ is a quasi-isomorphism.
\end{proof}
\subsection{The comparison theorem}
\begin{theorem}\label{thm:polhom}
  Let $C$ be a $1$\nbd-category. Then, for any $1
  \leq k \leq \oo$, there is a canonical isomorphism
  \[
    \Hop{k}(C) \to \Ho(C)
  \]
  natural in $C$. In particular, for $k=\oo$, we get an isomorphism
  $\Hopol(C) \to \Ho(C)$. Furthermore, for any $k\geq 1$, these
  isomorphisms make the following triangle commute
  \[
    \begin{tikzcd}[column sep=small]      
         \Hopol(C)\ar[dr]
         \ar[rr,"\homnat_C"]&&\Hop{k}(C)\ar[dl]\\
         &\Ho(C).&
    \end{tikzcd}
  \]
\end{theorem}

\begin{proof}
  Let us start with the first part of the theorem. Recall that if $F
  \colon \M \to \M'$ is a functor, and $C$ is a small category we
  denote by $\fun{C}{F} \colon \fun{C}{\M} \to \fun{C}{\M'}$ the
  functor induced by postcomposition. Consider the
  following zig-zag in $\Derp{\Z}$
  \[
    \begin{tikzcd}
      \hocolim_{C}p^*_C(\LL \abelk{k}(\Term))&&\\
    \hocolim_{C}\LL \fun{C}{\abelk{k}}(p_C^*\Term) \ar[u,"(a)"]&\ar[l,"(b)"'] \hocolim_{C}\LL\fun{C}{
     \abelk{k}}(\tr{C}{-})\ar[r,"(c)"]&  \LL
   \abelk{k}(\hocolim_{C}\tr{C}{-})\ar[d,"(d)"]\\
   && \LL\abelk{k}(C),
  \end{tikzcd}
  \]
  where
  \begin{itemize}
    \item $(a)$ and $(c)$ are the universal morphisms constructed in \ref{paragr:univmor},
    \item $(b)$ comes from the fact that $p_C^*\Term$ is the terminal
      object of $\funopCat{k}{C}$,
      \item $(d)$ is induced by the cocone $(\tr{C}{c}\to
  C)_{c \in \Ob(C)}$ (see \ref{paragr:colimtr}).
\end{itemize}
All these morphisms of $\Derp{\Z}$ are in fact isomorphisms:
\begin{itemize}
  \item $(a)$ by \eqref{eq:morphder} of \cref{prop:htpycocts} and the fact that $\abelk{k}$
    is a left Quillen functor (\cref{prop:abelquillen}),
      \item $(b)$ by \cref{prop:der2}, \cref{prop:hlgycontrcat} and the
    fact that for each object $c$ of $C$ the category $\tr{C}{c}$ has
    a terminal object,
  \item $(c)$ by \eqref{eq:cocontinuous} of \cref{prop:htpycocts} and the fact that $\abelk{k}$
    is a left Quillen functor,
  \item $(d)$ by \cref{prop:hocolimtr}.
  \end{itemize}
  Finally, since by
  \cref{paragr:abeldisk} and the fact that $\Dsk{0}=\Term$, we have
  \[
    \LL \abelk{k}(\Term) \cong \Z.
  \]
  Since, by definition, we have
  \[
    \Ho(C)=\hocolim_{C}p^*_C(\Z),
    \] this yields an isomorphism
  \[
    \Ho(C) \cong \Hop{k}(C),
  \]
  which is natural in $C$ by construction, as all the morphisms in the
  previous zig-zag are natural in $C$.

  For the second part, consider the following diagram in $\Derp{\Z}$
  \[
    \begin{tikzcd}
      \LL\abel(C)\ar[r] & \LL\abelk{k}(C)\\
       \LL\abel(\hocolim_{C}\tr{C}{-})\ar[r]\ar[u]&
       \LL\abelk{k}(\hocolim_{C}\tr{C}{-})\ar[u]\\
       \hocolim_{C}\fun{C}{\LL
     \abel}(\tr{C}{-})\ar[r] \ar[u]\ar[d]& \hocolim_{C}\fun{C}{\LL
     \abelk{k}}(\tr{C}{-})\ar[u]\ar[d]\\
   \hocolim_{C}\LL \fun{C}{\abel}(p_C^*\Term)\ar[r] \ar[d]& \hocolim_{C}\LL
   \fun{C}{\abelk{k}}(p_C^*\Term)\ar[d] \\
    \hocolim_{C}p^*_C(\LL \abel(\Term))\ar[r]& \hocolim_{C}p^*_C(\LL
    \abelk{k}(\Term)),
    \ar[from=1-1,to=2-2,phantom,description, "(A)"]
    \ar[from=2-1,to=3-2,phantom,description,"(B)"]
    \ar[from=3-1,to=4-2,phantom,description,"(C)"]
    \ar[from=4-1,to=5-2,phantom,description,"(D)"]
  \end{tikzcd}
\]
where the horizontal arrows are induced by the natural transformation
\eqref{transnat} from \cref{paragr:subtlety}
\[
    \homnat \colon \LL \abel\circ \oemb{k} \Rightarrow \LL \abelk{k}.
  \]
  Note that for simplicity we have not made the functor
  $\oemb{k} \colon \Loc{\opCat{k}}\to \Loc{\oCat}$ appear in the above
  diagram. This naturality implies that the squares $(A)$ and $(C)$ are
  commutative. Furthermore, the commutativity of the squares $(B)$ and
  $(D)$ follows from the naturality in $F$ of \eqref{eq:morphder} and
  \eqref{eq:cocontinuous} from \cref{prop:htpycocts}.

  To conclude, we only need to notice that the following triangle is
  commutative as an immediate computation shows
  \[
    \begin{tikzcd}
      \LL \abel(\Term) \ar[d,"\cong"']\ar[r] & \LL \abelk{k}(\Term)\ar[dl,"\cong"]\\
      \Z.&
    \end{tikzcd}
  \]
\end{proof}
\begin{corollary}
  For any $1$\nbd-category $C$ and any $1 \leq k \leq \oo$, the
  canonical natural morphism
  \[
   \homnat_C \colon \Hopol(C) \to \Hop{k}(C)
  \]
  is an isomorphism. If $C$ is a groupoid, this also holds for $k=0$.
\end{corollary}
\begin{remark}
  Adapting \emph{mutatis mutandis} the proof of
  \cref{thm:polhom}, one can prove more generally that, for a
  $1$-category $C$ and for any $1 \leq k \leq k' \leq \oo$, there is a canonical natural isomorphism
  \[
     \Hop{k'\!}(C)\cong  \Hop{k}(C).
  \]
\end{remark}
